\documentclass[12pt,twoside,a4paper,reqno]{amsart}
\usepackage{bulletinUPB_3}
\usepackage{graphics}
\usepackage{amssymb}
\usepackage{eucal}
\usepackage{amsmath}

\RequirePackage{fix-cm}
\usepackage{graphicx}
\usepackage{amsmath}
\DeclareMathOperator*{\argmin}{arg\,min}
\usepackage[utf8x]{inputenc}
\usepackage{cite}
\usepackage{url}
\usepackage{amsmath,amsfonts,amssymb}
\usepackage{fancyhdr}
\usepackage{algorithm}
\usepackage{algpseudocode}
\usepackage{accents}
\newcommand{\ubar}[1]{\underaccent{\bar}{#1}}
\usepackage{tkz-tab}
\usepackage{xpatch}
\usepackage{enumerate}   
\usepackage{multirow}
\usepackage{arydshln}
\usepackage{slashbox}
\usepackage{adjustbox}

\usepackage[title]{appendix}
\usepackage{verbatim}
\newcommand{\red}[1]{\textcolor{black}{#1}}

\newenvironment{MSC2000}
    {\par\noindent\textbf{Mathematics Subject Classification (2020):} }
    {\par}

\newtheorem{assumption}{Assumption}
\DeclareMathOperator{\sign }{sign}

\begin{document}

\info{A}{72}{1}{2010}

\title{Convergence analysis of linearized  $\ell_q$ penalty methods  for nonconvex  optimization with nonlinear equality constraints}

\author{Lahcen El Bourkhissi}
\address{$^1$ Department of Automatic Control and Systems Engineering, National University of Science and Technology  Politehnica Bucharest, 
             e-mail: {\tt lel@stud.acs.upb.ro}}
             
\author{Ion Necoara}
\address{$^2$    Department of Automatic Control and Systems Engineering, National University of Science and Technology  Politehnica Bucharest, 060042 Bucharest and Gheorghe Mihoc-Caius Iacob Institute of Mathematical Statistics and Applied Mathematics of the Romanian Academy, 050711 Bucharest, Romania.   e-mail: {\tt ion.necoara@upb.ro}}

\date{Received: date / Accepted: date}

\pagestyle{headings}
\maketitle

\begin{abstract}


{\it  \textbf{Abstract:}}
In this paper, we consider nonconvex optimization problems with nonlinear equality constraints. We assume that the objective function and the functional constraints are locally smooth. To solve this problem, we introduce a linearized $\ell_q$ penalty based method, where  $q \in (1,2]$ is the parameter defining the norm used in the construction of the penalty function. \red{Our method} involves linearizing the objective function and functional constraints in a Gauss-Newton fashion  at the current iteration in the penalty formulation and \red{introduces} a quadratic regularization. This approach yields an easily solvable subproblem, whose solution becomes the next iterate. \red{By using a novel dynamic rule for the choice of the regularization parameter, we establish that the iterates of our method converge to an $\epsilon$-first-order solution in $\mathcal{O}(1/{\epsilon^{2+ (q-1)/q}})$  outer iterations}. Finally, we put theory into practice and evaluate the performance of the proposed algorithm by making numerical comparisons with existing methods  from literature.
\end{abstract}

\begin{Keywords}
Nonconvex optimization, nonlinear functional constraints, $\ell_q$ penalty, linearized $\ell_q$ penalty, convergence analysis.
\end{Keywords}

\begin{MSC2000}
 68Q25 .  90C06 . 90C30.
\end{MSC2000}


\section{Introduction}
\label{intro}
\noindent In various fields, such as machine learning, matrix optimization, statistics, control and signal processing, a wide spectrum of applications can be reformulated as nonconvex optimization problems involving nonlinear functional equality constraints. References \cite{ LukSab:19,  HonHaj:17} serve as examples of this tendency. In our work, we propose an algorithmic framework  to solve such optimization problems based on  a penalty approach.

\vspace{0.2cm}

\noindent \textit{Related work:}
Penalty methods have played a central role in theoretical and numerical optimization, with their historical origins dating back to at least \cite{Cou:43}. Extensive research has explored the applications of penalty methods to a wide range of problems, as shown by works such as \cite{PolTre:73, Ber:76, KonMel:18, IzmSol:23, NocWri:06, Fle:87,  CarGou:11, LinMa:22}, among others. For example, \cite{IzmSol:23} studies a polyvalent class of penalty functions, taking the form of $\|\cdot\|_q^q$, where $q > 0$, for general constrained problems. This study establishes bounds that measure the closeness of the penalty solution to the solution of original problem as a function of the penalty parameter $\rho$. In particular,  under strict Mangasarian–Fromovitz constraint qualification and second-order sufficiency, a bound of the form $\mathcal{O}(\frac{1}{\rho^{q-1}})$ is derived and it becomes zero for $q\in(0,1]$ provided that $\rho$ is sufficiently large. Paper  \cite{CarGou:11} introduces an  algorithm based on  a Lipschitz penalty function, with dynamic quadratic regularization. This method reduces the size of a first-order criticality measure to a specified accuracy threshold $\epsilon$, in a maximum of $\mathcal{O}(1/\epsilon^2)$ functions evaluations, provided we are close to feasibility. In an alternative context, \cite{KonMel:18} focuses on the use of a quadratic penalty method to handle nonconvex composite problems with linear constraints proving convergence to an  $\epsilon$-critical  point in  $\mathcal{O}(1/\epsilon^3)$  accelerated composite gradient  steps. Furthermore, the work in \cite{LinMa:22} introduces an inexact proximal-point penalty method for solving general problems with nonconvex objective  and constraints, proving convergence to an $\epsilon$-critical point  within $\mathcal{O}(1/\epsilon^3)$ of functions evaluations. The result can be refined, reaching a  complexity of $\mathcal{O}(1/\epsilon^{2.5})$ for nonconvex objective and  convex constraints.  Finally, in our previous work \cite{ElbNec:25}, we developed a quadratic penalty method for solving smooth nonconvex optimization problems with nonlinear constraints, where we linearize both the objective and functional constraints within the quadratic penalty function in a Gauss-Newton fashion. We established a complexity bound of 
$\mathcal{O}(1/\epsilon^{2.5})$ of functions evaluations. In this context, our current method is inspired by \cite{IzmSol:23} and generalizes our previous work \cite{ElbNec:25}, as it considers an 
$\ell_q$ penalty approach with $q \in (1,2]$, bridging the gap between two extremes: the exact penalty method based on $\ell_1$ norm, where the penalty parameter $\rho$ is finite but the subproblem lacks differentiability, and the quadratic penalty method, where the subproblem is smooth but the penalty parameter \red{$\rho$} must be of the order inverse of the desired accuracy. Our method is intrinsically general and adapts to values of $q$, i.e.,  when $q=2$  our approach aligns with the quadratic penalty method \cite{ElbNec:25}, while when $q$ approaches $1$ we get closer to a Lipschitz penalty method (exact penalty), while retaining  differentiability of the subproblem.
 

\vspace{0.2cm}

\noindent \textit{Contributions:} Our approach, referred to as the linearized $\ell_q$ penalty method (qLP), effectively addresses  some of the  limitations of the previous  studies. Notably, in  \cite{CarGou:11}, \red{the subproblem is non-differentiable, due to the use of a Lipschitz penalty function,} while in \cite{KonMel:18}, the framework is limited to handling \red{only} linear constraints. \red{Hence, our main contributions are as follows}:

(i) At each iteration, we linearize, in a Gauss-Newton fashion, both the cost function and the nonlinear functional constraints within the $\ell_q$ penalty function, {where  $q \in (1,2]$ is the parameter defining the norm used in the construction of the penalty function,} and add a dynamic regularization term. This results in a new algorithm, called \textit{the linearized $\ell_q$ penalty method} (qLP). Notably, our method considerably simplifies the computational cost of the new iterate, since each iteration reduces to {minimizing a strongly convex differentiable function with Holder continuous gradient, making the subproblem easily solvable with e.g., an accelerated first-order scheme.}

(ii) We provide rigorous proofs of global asymptotic convergence, guaranteeing that the iterates eventually converge to a critical point of the $\ell_q$ penalty function, which implies, for an appropriate choice of $\rho$, a $(0,\epsilon)$-first-order solution of the original problem. Furthermore, our method guarantees convergence to an $\epsilon$-first-order solution of the original problem in $\mathcal{O}(1/\epsilon^{2+ (q-1)/q})$ outer iterations, thus improving the existing bounds.

(iii)  Compared to \cite{CarGou:11}, which employs a Lipschitz penalty function and has a total complexity of order \(\mathcal{O}\left(\epsilon^{-3}\right)\)  {when employing smoothing and an accelerated gradient scheme for solving the convex nonsmooth subproblem,} our approach exhibits a {better} total  complexity of order  $\mathcal{O}\left(1/\epsilon^{2+((q-1)/q)+(1/(3q-2))}\right)$
{when we \red{use} an accelerated gradient method for solving the corresponding  (strongly) convex subproblem with Holder gradient}, despite the complexity in evaluating Jacobians  being slightly more favorable for \cite{CarGou:11}. Still comparing the complexity of the subproblems, the algorithm in \cite{LinMa:22} is  difficult to implement in practice due to its high nonconvexity caused by the presence of nonlinear constraints in the subproblem from each iteration. Moreover, unlike \cite{KonMel:18}, our qLP method can handle general nonlinear equality constraints.  

\vspace{0.2cm}

\noindent The paper is organized as follows: In section \ref{sec2}, we begin by presenting the problem we \red{are} focusing on, as well as the essential concepts needed for our analysis. Then, in section \ref{sec3}, we present our algorithm, while Section \ref{sec4} is devoted to analyzing its convergence. Finally, Section \ref{sec5} is dedicated to a numerical comparison of our method with existing algorithms.


\section{Problem formulation and  preliminaries}
\label{sec2}
\noindent In this paper, we consider the following nonconvex optimization problem:
\begin{equation}
\begin{aligned}\label{eq1}
 \underset{x\in\mathbb{R}^n}{\min}
 f(x) \quad
 \textrm{s.t.}
 \quad F(x)=0,
\end{aligned}
\end{equation}
where $f:\mathbb{R}^{n}\to {\mathbb{R}}$ and  $F(x)\triangleq{(f_1(x),...,f_m(x))}^T$, with $f_i:\mathbb{R}^{n}\to {\mathbb{R}}$ for all $i=1:m$. We assume that the functions  $f, f_i \in \mathcal{C}^1$ for all  $i=1:m$,  where $f$  can be nonconvex  and $F$  nonlinear. {Moreover, we assume that the problem is well-posed i.e., the feasible set is nonempty and the optimal value is finite. Before introducing the main assumptions for our analysis, we would like to clarify some notations.  We use $\|\cdot\|_q^q$, where $q\in(1,2]$, to denote  the  $q$-norm of a vector in $\mathbb{R}^{n}$. For simplicity, $\|\cdot\|$ denotes  the Euclidean norm of a vector or the spectral norm of  a matrix.  For a differentiable function \( f: \mathbb{R}^n \to \mathbb{R} \), we denote by \( \nabla f(x) \in \mathbb{R}^n \) its gradient at a point \( x \). Moreover, we say that \( x^* \) is a \textit{critical} point of \( f \) if  \( \nabla f(x^*) = 0 \).
 For a differentiable vector function $F:\mathbb{R}^n  \to\mathbb{R}^m$, we denote its Jacobian at a given point $x$ by ${J}_F(x)\in\mathbb{R}^{m\times n}$. Furthermore, for a vector $y=(y_1,\ldots,y_m)^T\in\mathbb{R}^m$ and a positive value  $a$, we denote $|y|^a = (|y_1|^a,\ldots, |y_m|^a)^T$  and $\sign(y)\circ|y|^{a} = ( \sign(y_1) |y_1|^{a}, \cdots, \sign(y_m) |y_m|^{a} )^T \in\mathbb{R}^m$.  We further introduce the   notations:
 \[
 l_f(x;\bar{x})\triangleq f(\bar{x})+\langle\nabla f(\bar{x}),x-\bar{x}\rangle, \;\; l_F(x;\bar{x})\triangleq F(\bar{x})+ J_F(\bar{x})(x-\bar{x}) \hspace{0.5cm}\forall x,\bar{x}. 
 \]
Let us now present the main assumptions considered  for problem \eqref{eq1}:}
\begin{assumption}\label{assump1}
Assume that $f(x)$ has  compact level sets, i.e., for any $\alpha\in\mathbb{R}$, the following set is either empty or compact:
\[
\mathcal{S}_{\alpha}^0\triangleq{\{x:\, f(x)\leq\alpha\}}.
\]
\end{assumption}
\begin{assumption}\label{assump2}
Given a compact set $\mathcal{S}\subseteq\mathbb{R}^n$, there exist positive constants $M_f, M_F, \sigma, L_f, L_F$ such that $f$ and $F$ satisfy the following conditions:
\begin{enumerate}[(i)]
  \item $ \|\nabla f(x)\|\leq M_f, \hspace{0.5cm}\|\nabla f(x)-\nabla f(y)\|\leq L_f\|x-y\| \quad \red{\forall} x,y \in \mathcal{S}$.\label{ass1}
  \item $  \|J_F(x)\|\leq M_F, \hspace{0.5cm}\|J_F(x)-J_F(y)\|\leq L_F\|x-y\| \quad \forall x, y\in\mathcal{S}$. \label{ass2}
  \item Problem  \eqref{eq1} satisfies Linear Independance Constraint Qualification (LICQ) condition  \text{ for all } $x\in\mathcal{S}$\label{ass3}.
\end{enumerate}
\end{assumption}

\begin{assumption}\label{assump3}
There exists finite $\bar{\alpha}$ such that $f(x)\leq\bar{\alpha}$ for all $x\in\{x:\, \|F(x)\|\leq1\}$.
\end{assumption}
\noindent Note that these assumptions are standard in the nonconvex optimization literature, in particular in penalty type methods,  see e.g., \cite{CarGou:11,XieWri:21,CohHal:21,ElbNec:25}. In fact, these assumptions are not restrictive because they need to hold only  locally. Indeed,  large classes of problems satisfy these assumptions as discussed below.

\begin{remark}
Assumption \ref{assump1} holds e.g., when $f(x)$ is coercive ({in particular,} $f(x)$ is strongly convex),  or $f(x)$ is bounded {from} bellow. 
\end{remark}

\begin{remark}
Assumption \ref{assump2} allows general classes of problems. In particular, conditions \textit{(\ref{ass1})} hold if $f(x)$ is differentiable and $\nabla f(x)$ is \textit{locally}  Lipschitz continuous on a neighborhood of $\mathcal{S}$. Conditions \textit{(\ref{ass2})} hold when $F(x)$ is differentiable on a neighborhood of $\mathcal{S}$ and $J_F(x)$ is \textit{locally} Lipschitz continuous on $\mathcal{S}$. Finally, the LICQ assumption guarantees the existence of dual multipliers and is commonly used in nonconvex optimization, see e.g.,  \cite{NocWri:06, XieWri:21}. \red{For} equality constraints, LICQ holds on a compact set \( \mathcal{S} \) if the smallest singular value of the Jacobian matrix of the functional constraints remains strictly positive~on~\( \mathcal{S} \).
\end{remark}

\begin{remark}
For  Assumption \ref{assump3} to hold, it is sufficient the set $\{x: \,\|F(x)\|\leq 1\}$ to be compact. In fact, we do not need this assumption if we can choose the starting point $x_0$ such that $F(x_0)=0$, that is, the initial point is feasible.
\end{remark}

\noindent The following lemma is an immediate consequence of Assumption \ref{assump1}.
\begin{lemma} If  Assumption \ref{assump1} holds, then for any $\rho\geq 0$, we have:
\begin{equation}\label{lem1}
 \ubar{P}\triangleq{\inf_{x\in\mathbb{R}^n}\{ f(x)+\frac{\rho}{q}\|F(x)\|_q^q\}}>-\infty \quad \text{ and } \quad     \bar{f}\triangleq{\inf_{x\in\mathbb{R}^n}\{ f(x) \}}>-\infty.   
\end{equation}
\end{lemma}

\noindent We are interested in (approximate) first-order (also called KKT) solutions of \red{optimization problem} \eqref{eq1}.  Hence,   let us  introduce the following definitions: 
\begin{definition}\label{firstorder}[First-order solution  and $\epsilon$-first-order solution of \eqref{eq1}]
The vector $x^*$ is said to be a first-order solution of \eqref{eq1} if  $\exists\lambda^*\in\mathbb{R}^m$ such that: 
\begin{equation*}
\nabla f(x^*)+J_F(x^*)^T\lambda^*=0\hspace{0.3cm} \emph{and} \hspace{0.3cm} F(x^*)=0.
\end{equation*} 
Moreover, $\hat{x}$ is an $(\epsilon_1,\epsilon_2)$-first-order solution of \eqref{eq1} if  $\exists\hat{\lambda}\in\mathbb{R}^m$ and $\kappa_1, \kappa_2>0$ such that:
\begin{equation*}
    \|\nabla f(\hat{x})+J_F(\hat{x})^T\hat{\lambda}\|\leq\kappa_1\epsilon_1\hspace{0.3cm} \emph{and} \hspace{0.3cm} \|F(\hat{x})\|\leq\kappa_2\epsilon_2.
\end{equation*} 
\noindent \red{If $\epsilon_1=\epsilon_2$, we refer to $\hat{x}$ as an $\epsilon$-first-order solution in the previous definition.}
\end{definition}


\section{A linearized $\ell_q$ penalty method}\label{sec3}
\noindent In this section, we propose a new algorithm for solving nonconvex problem \eqref{eq1} using the $\ell_q$ penalty framework. Let us first introduce few  notations. The penalty function associated with the problem \eqref{eq1} is 
\begin{equation}\label{penalty_function}
    \mathcal{P}^q_{\rho}(x)=f(x)+\frac{\rho}{q}{\|F(x)\|_q^q},
\end{equation}
where $q \!\in\! (1,2]$. This penalty function,  $\mathcal{P}^q_{\rho}$,  is differentiable and its gradient   is: 
\[\nabla\mathcal{P}^q_{\rho}(x)={\nabla f(x)+J_F(x)}^T\left(\rho \sign(F(x))\circ|F(x)|^{q-1}\right).\] 
In the next lemma, we show that function $\sign(\cdot)|\cdot|^{\nu}$, where $\nu \in (0, 1]$, satisfies the Holder condition (\red{in the rest of the paper, for the sake of clarity, we provide the proofs of all the lemmas in Appendix}):
\begin{lemma}
\label{holder}[Holder] 
Let $\nu\in(0,1]$. Then, we have:
\[
|\sign(x)|x|^{\nu}-\sign(y)|y|^{\nu}|\leq 3|x-y|^{\nu} \quad \forall x,y\in\mathbb{R}.
\]
\end{lemma}
\begin{proof}
See Appendix.
\end{proof}

\noindent From \red{the} previous lemma, using properties of norms, one can conclude that the function $v \mapsto \frac{1}{q} \|v\|_q^q$, with $q \in (1, 2]$ and $v\in\mathbb{R}^m$, has {the gradient Holder continuous w.r.t. Euclidean norm  $\|\cdot\|$}, i.e.,:
\begin{equation} \label{holder_inequality0}
\| \sign(v)\circ|v|^{q-1} - \sign(w)\circ|w|^{q-1} \| \leq 3 \times m^{\frac{2-q}{2}} \| v- w\|^{q-1} \quad \forall v,w\in\mathbb{R}^m.  
\end{equation}
{Indeed, we have:
\begin{align*}
   &\resizebox{\textwidth}{!}{$\| \sign(v)\circ|v|^{q-1} - \sign(w)\circ|w|^{q-1} \| = \left[ \sum_{i=1}^m \left( \left|\sign(v_i)|v_i|^{q-1} - \sign(w_i)|w_i|^{q-1}\right| \right)^2\right]^{\frac{1}{2}}$}\\
  &\leq 3  \left[ \sum_{i=1}^m \left(\left|v_i-w_i\right|^2\right)^{q-1}\right]^{\frac{1}{2}}  = 3  \left[ \sum_{i=1}^m \left(\left|v_i-w_i\right|^2\right)^{q-1} \times 1^{2-q} \right]^{\frac{1}{2}}\\
  & \leq 3 \left(\sum_{i=1}^m \left|v_i-w_i\right|^2\right)^\frac{q-1}{2} \times \left(\sum_{i=1}^m 1\right)^{\frac{2-q}{2}} = 3\times m^{\frac{2-q}{2}}\|v-w\|^{q-1},
\end{align*}
where the last inequality follows from the Holder inequality, i.e., for any $x,y\in\mathbb{R}^m$ and for any 
\( r, s \in \mathbb{R}_{+} \) the following holds:
\[
\left( \sum_{i=1}^{m} |x_i|^r |y_i|^s \right)^{r+s} 
\leq 
\left( \sum_{i=1}^{m} |x_i|^{r+s} \right)^r 
\left( \sum_{i=1}^{m} |y_i|^{r+s} \right)^s.
\]
}

\noindent Relation \eqref{holder_inequality0} implies the following inequality \cite{DevGli:13}:
\begin{equation} \label{holder_inequality}
\textstyle
\frac{1}{q}\|v\|_q^q \leq \frac{1}{q}\|w\|_q^q + \langle \sign(w)\circ|w|^{q-1}, v-w\rangle + \frac{3\times m^{\frac{2-q}{2}}}{q} \| v- w\|^{q}, \quad \forall v,w\in\mathbb{R}^m.
\end{equation}

\noindent Further, let us denote the following function derived from linearization \red{in a Gauss-Newton fashion}  of  the objective function and the functional constraints, at a given point $\bar{x}$, in the penalty function:
\begin{align*}
 \bar{\mathcal{P}}^q_{\rho}(x;\bar{x})=f(\bar{x})+\langle\nabla f(\bar{x}),x-\bar{x}\rangle +\frac{\rho}{q}{\|F(\bar{x})+J_F(\bar{x})(x-\bar{x})\|_q^q}.   
\end{align*}  
Note that the function $\bar{\mathcal{P}}^q_{\rho}(\cdot;\bar{x})$ is always convex since $q>1$. Let us also introduce the following criticality measure for the penalty function $\mathcal{P}^q_{\rho}$, for conducting our analysis, inspired by \cite{CarGou:11}. \red{For} $0 < r \leq 1$, we define: 
\begin{equation}
\label{eq:Psi}
\Psi_r(x) \triangleq  \bar{\mathcal{P}}^q_{\rho}(x;x) - \min_{\|y-x\|\leq r}  \bar{\mathcal{P}}^q_{\rho}(y;x) = \mathcal{P}^q_{\rho}(x) - \min_{\|y-x\|\leq r}  \bar{\mathcal{P}}^q_{\rho}(y;x).
\end{equation}
In particular, following \cite{ Yua:85}, $\Psi_r(x)$ is continuous for all $x$, and  $x^*$ is a critical point of \red{penalty function}  $\mathcal{P}^q_{\rho}$ if
\begin{equation}
\label{eq:critical_point}
\Psi_r(x^*) = 0.
\end{equation}
\red{In the next lemma we prove the above claim, see also Lemma 2.1 in \cite{Yua:85}.}

\begin{lemma}\label{yuan85} Let $ q\in (1,2]$,  $0<r\leq1$ and $\Psi_r(\cdot)$ be as in  \eqref{eq:Psi}. Then,   $\Psi_r(x) \geq 0$ and $\Psi_r(x) = 0$ if and only if $x$ is a critical point of the penalty function  $\mathcal{P}^q_{\rho}$. Moreover,  $\Psi_r(\cdot)$ is continuous.
\end{lemma}

\begin{proof}
See appendix.
\end{proof}

\noindent Let us also introduce the following pseudo-criticality measure: 
\begin{equation}
\label{eq:Psi_r}
\begin{aligned}
\bar{\Psi}(x, \beta) &\triangleq \bar{\mathcal{P}}^q_{\rho}(x;x) - \min_{y \in \mathbb{R}^n} \left\{ \bar{\mathcal{P}}^q_{\rho}(y;x)+ \frac{\beta} {2}\|y-x\|^2 \right\}. \\
\end{aligned}
\end{equation}
We establish later a relation between these two criticality measures ${\Psi}_r(x)$ and  $\bar{\Psi}(x,\beta)$, respectively.

\medskip 

\noindent To solve the optimization problem \eqref{eq1} we propose the following \textit{Linearized  $\ell_q$ penalty} (qLP) algorithm, \red{where} we linearize the objective function and the functional constraints, in a Gauss-Newton fashion, within the penalty function at the current iterate and add an \textit{adaptive} quadratic regularization.
\begin{algorithm}
\caption{Linearized $\ell_q$ penalty (qLP) method}\label{alg1}
\begin{algorithmic}[1]
\State  $\textbf{Initialization: } x_0, \rho>0, \text{ and } \ubar{\beta} \geq 1$.
\State $k \gets 0$
\While{$\text{ stopping criterion is not satisfied }$}
    \State $\text{generate a proximal parameter } \beta_{k+1}\geq\ubar{\beta}$ such that
    \State $x_{k+1}\gets\argmin_{x\in\mathbb{R}^n}{\bar{\mathcal{P}}^q_{\rho}(x;x_{k})+\frac{\beta_{k+1}}{2}{\|x-x_{k}\|}^2}$ satisfies the descent:
    \begin{equation} \label{smoothness}
     \mathcal{P}^q_{\rho}(x_{k+1}) \leq \bar{\mathcal{P}}^q_{\rho}(x_{k+1};x_k) + \frac{\beta_{k+1}}{2} \|x_{k+1} - x_k\|^2.
\end{equation}
    \State $k \gets k+1$
\EndWhile
\end{algorithmic}
\end{algorithm}
\noindent \textit{To the best of our knowledge qLP algorithm is new and its convergence behavior has not been analyzed before in the literature.}  Note that  the objective function in the subproblem of Step 5 of  Algorithm \ref{alg1} is always strongly convex  since the convex function $\bar{\mathcal{P}}^q_{\rho}(\cdot;{x}_k)$ is regularized with a quadratic term. {Moreover, it has a locally Holder continuous gradient with exponent $q-1$ (see \eqref{holder_inequality0} for the Holder continuity of the gradient of the term $v \mapsto \|v\|_q^q$ and note that the quadratic term has also a Holder continuous gradient on any compact subset of $\mathbb{R}^n$). Indeed, the gradient of the subproblem objective function is:
\begin{align*}
&\nabla_x\left({\bar{\mathcal{P}}^q_{\rho}(x;x_{k})+\frac{\beta_{k+1}}{2}{\|x-x_{k}\|}^2}\right) \\
&= \nabla f(x_k) + \rho J_F(x_k)^T\sign(l_F(x;x_k))\circ|l_F(x;x_k)|^{q-1} + \beta_{k+1} (x-x_k).
\end{align*}
Hence,
\begin{align*}
&\left\|\nabla_x\left({\bar{\mathcal{P}}^q_{\rho}(x;x_{k})+\frac{\beta_{k+1}}{2}{\|x-x_{k}\|}^2}\right) - \nabla_x\left({\bar{\mathcal{P}}^q_{\rho}(y;x_{k})+\frac{\beta_{k+1}}{2}\|y-x_{k}\|^2}\right)\right\| \\
&\resizebox{\textwidth}{!}{$\leq \rho \|J_F(x_k)\| \cdot \left\|\sign(l_F(x;x_k))\circ|l_F(x;x_k)|^{q-1}-\sign(l_F(y;x_k))\circ|l_F(y;x_k)|^{q-1}\right\| + \beta_{k+1}\|x-y\|$}\\
&\overset{\eqref{holder_inequality0}}{\leq} 3\times m^{\frac{2-q}{2}}\rho\|J_F(x_k)\| \cdot \|l_F(x;x_k)- l_F(y;x_k)\|^{q-1} + \beta_{k+1}\|x-y\|\\
& = 3\times m^{\frac{2-q}{2}}\rho\|J_F(x_k)\| \cdot \|J_F(x_k)(x-y)\|^{q-1} + \beta_{k+1}\|x-y\|^{2-q}\|x-y\|^{q-1}\\
&\leq \left(3\times m^{\frac{2-q}{2}}\rho M_F^q + \beta_{k+1}\|x-y\|^{2-q}\right)\|x-y\|^{q-1}.
\end{align*}
Thus, on any compact set $\mathcal{S}\subset\mathbb{R}^n$ with diameter $D_{\mathcal{S}}$, the the objective function in the subproblem of Step 5 of Algorithm \ref{alg1}, which is unconstrained, has a Holder continuous gradient with constant $\left(3\times m^{\frac{2-q}{2}}\rho M_F^q + \beta_{k+1}D_{\mathcal{S}}^{2-q}\right)$ and exponent $q-1$. Assuming that the iterates of Algorithm \ref{alg1} are bounded and the bound is independent of the algorithm's parameters,  using a gradient descent type method to solve the subproblem, one can still show that the inner iterates remain within the compact set where the outer iterates belong. Hence, the diameter $D_{\mathcal{S}}$ of the compact set $\mathcal{S}$ containing the inner/outer iterates is independent on $\rho$.} Therefore, finding a solution of the subproblem in Step 5 is easy as there are efficient gradient descent type methods that can minimize this type of smooth objective function (see e.g., \cite{DevGli:13}).   In the sequel, we denote: 
\[
\Delta x_{k}=x_{k}-x_{k-1} \hspace{0.3cm} \forall k\geq1.
\]
\noindent Let us show that we can always choose  {an adaptive regularization parameter} \(\beta_{k+1}\) guaranteeing  the descent \red{property} \eqref{smoothness}. Indeed, since $f$ and $F$ are smooth functions, if one chooses \red{adaptively (i.e., depending on the current iterate $x_k$)}:

\vspace{-0.3cm}
 
\begin{equation} \label{eq_assu}
    \beta_{k+1} \geq L_f +  \left(3\times m^{\frac{(2-q)(q-1)}{2q}}\times 2^{2-q} \right)^{\frac{1}{q}}q^{\frac{q-1}{q}}\rho^{\frac{1}{q}}L_F\left(\mathcal{P}^q_{\rho}(x_k)-\bar{f}\right)^{\frac{q-1}{q}},
\end{equation} 
\red{then} the  descent property \eqref{smoothness} follows, as established in the following lemma.  

\begin{lemma}\label{lemma3}[Existence of $\beta_{k+1}$]  If  the sequence $\{x_{k}\}_{k\geq0}$ generated by Algorithm \ref{alg1} is in some compact set $\mathcal{S}$ on which Assumptions \ref{assump1} and \ref{assump2} hold and we choose $\beta_{k+1}$  as in \eqref{eq_assu}, then the descent property \eqref{smoothness} holds. Consequently, the following decrease condition is also satisfied:
  \begin{align}\label{eq_descent}
  \mathcal{P}^q_{\rho}(x_{k+1})\leq \mathcal{P}^q_{\rho}(x_k) - \frac{\beta_{k+1}}{2}\|x_{k+1}-x_{k}\|^2.
  \end{align}
\end{lemma}
\begin{proof}
See Appendix.
\end{proof}
 \noindent Note that the choice of \( \beta_{k+1} \) in \eqref{eq_assu} is inspired by \cite{MarOku:24,ElbNec:25}, where \( q \) is chosen as \( q=2 \). The proof of the above lemma follows similar arguments as in \cite{MarOku:24,ElbNec:25}, with this paper extending the analysis to \( q \in (1,2] \). Lemma \ref{lemma3} establishes that the qLP algorithm is implementable since choosing \( \beta_{k+1} \) as in \eqref{eq_assu} already ensures \eqref{smoothness} and \eqref{eq_descent}. In practice, the regularization parameter \( \beta_k \) can be determined using a backtracking scheme, as described in Algorithm 2 in \cite{ElbNec:25}.   Note that in a usual approach, based on \eqref{holder_inequality}, one uses the descent \eqref{smoothness} with $q$ regularization  instead of quadratic regularization and $\beta_{k+1}$ of order  $\mathcal{O}(\rho)$, which will result in a  worse complexity. However,  using  a quadratic regularization in \eqref{smoothness} and a novel choice of $\beta_{k+1}$ (see \eqref{eq_assu}),  allow us to derive improved rates in this paper.  {Next, we prove that  $x_{k+1}$   guarantees the following:}
 
\begin{lemma} \label{bounded_regularization}
 Let Assumption \ref{assump2} hold on a compact set $\mathcal{S}$ and assume that the sequence $\{x_{k}\}_{k\geq0}$ generated by Algorithm \ref{alg1} is in $\mathcal{S}$. Then, we have:
\begin{equation}
\label{succesfull}
 r_k \triangleq \frac{\mathcal{P}^q_{\rho}(x_k) - \mathcal{P}^q_{\rho}(x_{k+1})}{\bar{\Psi}(x_k, \beta_{k+1})}\geq 1 \quad {\forall k \geq 0}. 
 \end{equation}
\end{lemma}
\begin{proof} 
 See appendix.
\end{proof}


\section{Convergence analysis}
\label{sec4}
\noindent In this section, we first explore the asymptotic convergence of the qLP algorithm (Algorithm \ref{alg1}) and then derive its efficiency in finding an $\epsilon$-first-order solution for the problem \eqref{eq1}. Our analysis combines Lyapunov analysis techniques, inspired by \cite{ElbNec:25,XieWri:21}, with criticality measure analysis, inspired by \cite{CarGou:11}. In the sequel, we are using the $\ell_q$ penalty function, $\mathcal{P}^q_{\rho}$,  as a Lyapunov function.  The evaluation of the Lyapunov function along the iterates of qLP algorithm  is denoted by:
\begin{equation}\label{lyapunov_function}
 P_{k}=\mathcal{P}^q_{\rho}(x_k) \hspace{0.3cm} \forall k\geq0.
\end{equation}
It is clear, from Lemma \ref{lemma3},  that $\{P_{k}\}_{k\geq0}$ is decreasing  and later we prove that it is bounded from bellow.  In the sequel, we assume that $x_0$ is chosen \red{as}:
\begin{equation}\label{eq4}
    \|F(x_0)\|_q^q\leq\min\left\{1,\frac{qc_0}{\rho}\right\} \hspace{0.7cm} \text{ for some } c_0>0,
\end{equation}
and that $f(x_0)\leq\bar{\alpha}$. Let us define:
\begin{equation}\label{alpha_hat}
    \bar{P}\triangleq\bar{\alpha}+c_0.
\end{equation}
 Moreover, let us choose:
\begin{align}\label{rho_def}
 \rho\geq\max\left\{\!{1}, (q+1)\rho_0\right\}.   
\end{align}
Using the definition of $\mathcal{P}^q_{\rho}$, we have: 
 \begin{align}
    \mathcal{P}^q_{\rho}(x_0)&=f(x_0)+\frac{\rho}{q}\|F(x_0)\|_q^q{\overset{{\eqref{eq4}}}{\leq}}\bar{\alpha}+c_0 \label{ine10}.
\end{align} 
It then follows, after some re-arrangements, that:
\begin{align}
   \bar{f}+c_0-\ubar{P}&\geq f(x_0)+\frac{\rho}{q}\|F(x_0)\|_q^q-\ubar{P}\nonumber\\
   &{\overset{{(\rho\geq(q+1)\rho_0)}}{\geq}}f(x_0)+\frac{\rho_0}{q}\|F(x_0)\|_q^q-\ubar{P}+\frac{\rho}{q+1}\|F(x_0)\|_q^q\nonumber\\
   &{\overset{{\eqref{lem1}}}{\geq}} \frac{\rho}{q+1}\|F(x_0)\|_q^q\geq 0.\label{ine11}
\end{align}

\noindent The following lemma shows that if the sequence $\{x_{k}\}_{k\geq0}$ generated by  Algorithm \ref{alg1} is bounded, then   the Lyapunov sequence $\{P_{k}\}_{k\geq0}$ is also bounded.

\begin{lemma}
\label{bbound} 
Consider  Algorithm \ref{alg1} and let $\{P_{k}\}_{k\geq0}$ as defined in  \eqref{lyapunov_function}. If  the sequence $\{x_{k}\}_{k\geq0}$ generated by Algorithm \ref{alg1} is in some compact set $\mathcal{S}$ on which Assumptions \ref{assump1}, \ref{assump2} and \ref{assump3} hold and, moreover, $\rho$ is chosen as in \eqref{rho_def} and $x_0$ is chosen as in \eqref{eq4}, then we have the following:
\begin{equation}\label{important}
    \ubar{P}  \leq P_{k}\leq \bar{P}  \hspace{0.5cm}\forall k\geq0, 
\end{equation}
where $\ubar{P}$ is  defined in \eqref{lem1} and $\bar{P}$ defined in \eqref{alpha_hat} for any fixed constant $c_0$.
\end{lemma}
\begin{proof}
See Appendix.
\end{proof}
\noindent Let us bound the gradient of $\ell_q$ penalty function. Denote $C=\frac{3m^{\frac{2-q}{2}}\rho M_F L_F^{q-1}}{2^{q-1}}$.
\begin{lemma}
\label{bounded_gradient}
[Boundedness of $\nabla\mathcal{P}^q_{\rho}$]
Let  Assumption  \ref{assump2}  hold on a compact set $\mathcal{S}$ and  the sequence generated by  Algorithm \ref{alg1} satisfies $\{x_{k}\}_{k\geq0} \subseteq \mathcal{S}$. Then:
\[
    \|\nabla\mathcal{P}^q_{\rho}(x_{k+1})\|\leq C\|\Delta x_{k+1}\|^{2(q-1)}+\Gamma_{k+1}\|\Delta x_{k+1}\|,
\]
where $\Gamma_{k+1}=L_f+ m^{\frac{2-q}{2}}q^{\frac{q-1}{q}} L_F \rho^{\frac{1}{q}}\left(\mathcal{P}^q_{\rho}(x_{k}) - \bar{f}\right)^{\frac{q-1}{q}} +\beta_{k+1} \quad  \forall k\geq 0$.
\end{lemma}
\begin{proof}
See Appendix.
\end{proof}

\noindent From Lemma \ref{lemma3} it follows  that when using a backtracking scheme, with a geometrically increasing parameter $\mu>1$, $\beta_{k+1}$ can be always upper \red{bounded~as}:
\begin{equation} \label{bar_gamma}
\small 
\bar{\beta} \triangleq \sup_{k\geq 1} \beta_k \leq \mu \left(L_f + \left(3\times m^{\left(\frac{(2-q)(q-1)}{2q}\right)}\times 2^{2-q}\right)^{\frac{1}{q}} q^{\frac{q-1}{q}}\rho^{\frac{1}{q}}L_F (\bar{P}-\bar{f})^{\frac{q-1}{q}}\right).
\end{equation}

\noindent Moreover, since $\{\beta_k\}_{k\geq1}$ is bounded, then $\{\Gamma_k\}_{k\geq1}$ is also bounded. In the sequel, we define this bound as follows:
\begin{equation}\label{boundgam}
    \bar{\Gamma}\triangleq \sup_{k\geq1} \Gamma_k \leq L_f+ m^{\frac{2-q}{2}}q^{\frac{q-1}{q}} L_F \rho^{\frac{1}{q}}\left(\bar{P} - \bar{f}\right)^{\frac{q-1}{q}} +\bar{\beta} .
\end{equation}


\subsection{Global asymptotic convergence} 
\noindent In this section we prove global convergence for the iterates generated by  Algorithm \ref{alg1} and also convergence rates to an $\epsilon$-first-order solution. Based on the previous lemmas, we are now ready  to present the global asymptotic convergence of the iterates of  Algorithm \ref{alg1}.
\begin{theorem}\label{unused_lemma}[Limit points] If  the sequence, $\{x_{k}\}_{k\geq0}$, generated by Algorithm \ref{alg1} is in some compact set $\mathcal{S}$ on which Assumptions \ref{assump1}, \ref{assump2} and \ref{assump3} hold and, moreover, $\rho$ is chosen as in \eqref{rho_def} and $x_0$ is chosen as in \eqref{eq4}, then any limit point $x_{\rho}^*$ of the sequence $\{x_{k}\}_{k\geq0}$ is a critical point of the penalty function $\mathcal{P}^q_{\rho}(\cdot)$ defined in \eqref{penalty_function}, i.e., $\nabla\mathcal{P}^q_{\rho}(x_{\rho}^*)=0$. Moreover: 
\[
\nabla f(x_{\rho}^*)+{J_F(x_{\rho}^*)}^T\lambda^*=0,\hspace{0.3cm}\text{where} \hspace{0.3cm}  \lambda^* =\rho\sign(F(x_{\rho}^*))\circ|F(x_{\rho}^*)|^{q-1}.
\]
Furthermore, if we select $\rho$ such that $\rho\geq \mathcal{O}\left(\frac{1}{\epsilon^{q-1}}\right)$, we can also establish that $\|F(x^*_{\rho})\|\leq\mathcal{O}(\epsilon)$. In this context, we \red{get that  $x^*_{\rho}$ is an $(0,\epsilon)$-first-order solution of} problem \eqref{eq1}.
\end{theorem}
\begin{proof}
Using Lemma \ref{lemma3} and the fact that $\beta_{k}>0$, we have:
\begin{align*}
     \frac{\beta_{k+1}}{4}\|&\Delta x_{k+1}\|^2\leq P_{k}-P_{k+1}\hspace{0.3cm}\forall k\geq0.
\end{align*}
Let $k\geq0$, by summing up the above inequality from $i=0$ to $i=k$, we obtain:
\begin{align}
\sum_{i=0}^{k}{\frac{\beta_{i+1}}{4}\|\Delta x_{i+1}\|^2}&\leq P_{0}-P_{k+1}{\overset{{}}{\leq}} \bar{P}-\ubar{P}.\label{limit}
\end{align}
Since \eqref{limit} holds for any $k\geq0$, we have:
\[
\sum_{i=1}^{\infty}{\frac{\beta_{i+1}}{4}\|\Delta x_{i+1}\|^2}<\infty.
\]
This, together with the fact that $\beta_{k}>\ubar{\beta}\geq 1$, yields that:
\begin{equation}\label{zero_limit}
    \lim_{k\to\infty}{\|\Delta x_{k}\|}=0.
\end{equation}
Since the sequence $\{x_{k}\}_{k\geq0}$ is bounded, then there exists a convergent subsequence, let us say  $\{x_{k}\}_{k\in\mathcal{K}}$, with the limit $x_{\rho}^*$.
From {Lemma \ref{bounded_gradient}} and \eqref{boundgam}, we have:
\[
\|\nabla\mathcal{P}^q_{\rho}(x_{\rho}^*)\|=\lim_{k\in\mathcal{K}}{\|\nabla\mathcal{P}^q_{\rho}(x_{k})\|}\leq C\lim_{k\in\mathcal{K}}\|\Delta x_{k}\|^{2(q-1)}+\bar{\Gamma}\lim_{k\in\mathcal{K}}\|\Delta x_{k}\|=0.
\]
Therefore, $ \nabla\mathcal{P}^q_{\rho}(x_{\rho}^*)=0$, which means that there exists $\lambda^*=\rho\sign(F(x_{\rho}^*))\circ|F(x_{\rho}^*)|^{q-1}$ such that:
\[
\nabla f(x_{\rho}^*)+{J_F(x_{\rho}^*)}^T\lambda^*=0.
\]
Now, let's consider that $\rho\geq\mathcal{O}(\frac{1}{\epsilon^{q-1}})$ and prove that $\lambda^*=\rho\sign(F(x_{\rho}^*))\circ|F(x_{\rho}^*)|^{q-1}$ is bounded. Indeed, if we consider a KKT point $x^*\in\mathcal{S}$, LICQ ensures the existence of a corresponding finite $y^*$ that satisfies the following KKT conditions for problem \eqref{eq1}:
\[
\nabla f(x^*)+{J_F(x^*)}^Ty^*=0, \hspace{0.3cm}\text{and}\hspace{0.3cm} F(x^*)=0.
\]
It then follows that:
\[
y^*=-{J_F(x^*)}^{+}\nabla f(x^*),
\]
where ${J_F(x^*)}^{+}$ denotes the pseudo-inverse of matrix $J_F(x^*)$. Moreover:
\[
\nabla f(x_{\rho}^*)+{J_F(x_{\rho}^*)}^T\lambda^*=0,\hspace{0.3cm}\text{where} \hspace{0.3cm}   \lambda^*=\rho \sign(F(x_{\rho}^*))\circ |F(x_{\rho}^*)|^{q-1},
\]
we arrive at the relation:
\[
\lambda^*=-{J_F(x_{\rho}^*)}^{+}\nabla f(x_{\rho}^*).
\]
Subsequently, it follows that:
\[
\|\lambda^*-y^*\|=\|{J_F(x^*)}^{+}\nabla f(x^*)-{J_F(x_{\rho}^*)}^{+}\nabla f(x_{\rho}^*)\|.
\]
Given continuity of $J_F^+$ and $\nabla f$, along with the fact that $x^*$ and $x_{\rho}^*$ belong to the compact set $\mathcal{S}$, we conclude that there exists  $M\geq 0$ such that:
\[
\|\lambda^*-y^*\|\leq M.
\]
Hence:
\[
\|\lambda^*\|\leq M +\|y^*\|.
\]
Therefore, from the definition of $\lambda^*$, we get:
\[
\|F(x^*_{\rho})\|=\frac{\||\lambda^*|^{\frac{1}{q-1}}\|}{\rho^{\frac{1}{q-1}}}\leq \left(\frac{\|\lambda^*\|_q}{\rho}\right)^{\frac{1}{q-1}}\leq \mathcal{O}(\epsilon).
\]
This completes our proof.\qed
\end{proof}
\begin{remark}\label{remark_impo}
Note that the previous results apply to any value of $q\in(1,2]$. However, when considering the selection of $\rho$ such that $\rho\geq\mathcal{O}\left(\frac{1}{\epsilon^{q-1}}\right)$, it becomes evident that opting for $q$ close to 1 is the optimal choice. This is because the penalty parameter $\rho$ is smaller when $q$ is large and vice versa.
\end{remark}


\subsection{\red{Global convergence rates}}
\noindent Let us now investigate the computational complexity  of Algorithm \ref{alg1} for generating an $\epsilon$-first-order solution. First, we relate the model decrease $\bar{\Psi}(x_k, \beta_{k+1})$ to the optimality measure $\Psi_r(x_k)$ in \eqref{eq:Psi}.
\begin{lemma} \label{eq:lemma_2.5} 
Let $0<r\leq1$ and let $\Psi_r(\cdot)$ be defined by \eqref{eq:Psi}. If  the sequence, $\{x_{k}\}_{k\geq0}$ generated by Algorithm \ref{alg1} is in some compact set $\mathcal{S}$ on which Assumption \ref{assump2} holds, then:
\begin{equation} \label{coraliaa}
\bar{\Psi}(x_k, \beta_{k+1}) \geq \frac{1}{2} \min\left(1, \frac{\Psi_r(x_k)}{\beta_{k+1}r^2}\right) \Psi_r(x_k).
\end{equation}
\end{lemma}
\begin{proof}
See appendix.
\end{proof}

\noindent Lemma \ref{eq:lemma_2.5} indicates that $r_k$ in \eqref{succesfull} is well-defined whenever the current iteration is not a first-order critical point, i.e., $\Psi_r(x_k) \neq 0$. Let $0<\epsilon\leq 1$. 
The following theorem demonstrates that after $K\geq\mathcal{O}(\frac{\rho^{\frac{1}{q}}}{\epsilon^{2}})$ iterations of Algorithm \ref{alg1}, the criticality measure $\Psi_{\epsilon}(x_k)\leq \epsilon^2$.
\begin{lemma} \label{lemma_criticality}
Consider Algorithm \ref{alg1} and let $\{P_{k}\}_{k\geq0}$ be defined as in \eqref{lyapunov_function}. If  the sequence $\{x_{k}\}_{k\geq0}$, generated by Algorithm \ref{alg1} is in some compact set $\mathcal{S}$ on which Assumptions \ref{assump1}, \ref{assump2} and \ref{assump3} hold and, moreover, $\rho$ is chosen as in \eqref{rho_def} and $x_0$ is chosen as in \eqref{eq4}.  Then, for any $\epsilon \in (0, 1]$, 
after \[
K = \left\lceil2\bar{\beta}\left(\bar{P}-\ubar{P}\right) \epsilon^{-2} \right\rceil= \mathcal{O}\left(\frac{\rho^{\frac{1}{q}} }{\epsilon^2}\right)
\] iterations of Algorithm \ref{alg1}, we obtain  $\Psi_{\epsilon}(x_k)\leq \epsilon^2$.
\end{lemma}
\begin{proof} 
See Appendix.
\end{proof}

\noindent In the next theorem, we prove that when the optimality measure $\Psi_{\epsilon}(x_k)$ is sufficiently small, then we have an approximate critical point of the penalty function $\mathcal{P}^q_{\rho}(\cdot)$ \red{and and an $\epsilon$-first-order solution for the problem \eqref{eq1}}.

\begin{theorem}
If  the sequence, $\{x_{k}\}_{k\geq0}$ generated by Algorithm \ref{alg1} is in some compact set $\mathcal{S}$ on which Assumptions \ref{assump1}, \ref{assump2} and \ref{assump3} hold and, moreover, $\rho$ is chosen as in \eqref{rho_def} and $x_0$ is chosen as in \eqref{eq4} and  let $x_k$ be an iterate satisfying $\Psi_{\epsilon}(x_k) \leq \epsilon^2$ for a given tolerance $0<\epsilon \leq 1$.
Then, there exists $\lambda_k$ such that
\begin{equation} \label{kktt1}
   \|\nabla f(x_k) + J_F(x_k)^T \lambda_k\| \leq \epsilon. 
\end{equation}
Moreover, if $\rho = \mathcal{O}\left(\frac{1}{\epsilon^{q-1}}\right)$, then $x_k$ is an $\epsilon$-first-order solution for \eqref{eq1},  within $k = \mathcal{O}\left(\frac{1}{\epsilon^{2+\frac{q-1}{q}}}\right)$ iterations.
\end{theorem}

\begin{proof}
 Let us denote: 
\begin{equation*}
    s_k^* \!=\! \arg \min_{\|s\| \leq \epsilon} \bar{\mathcal{P}}^q_\rho(x_k+s;x_k) \!=\! \arg \min_{\|s\| \leq \epsilon} f(x_k) + \langle\nabla f(x_k), s\rangle + \frac{\rho}{q}\|F(x_k)+J_F(x_k)s\|_q^q.
   \end{equation*}
Assume that we are in the case $\|s_k^*\| < \epsilon$. Then  the above problem is essentially unconstrained and convex, and first-order conditions provide that $\nabla \bar{\mathcal{P}}^q_\rho(x_k+s_k^*;x_k)=0$, and so there exists $\lambda_k = \rho\sign(F(x_k)+J_F(x_k)s_k^*) \circ|F(x_k)+J_F(x_k)s_k^*|^{q-1}$ such that $\nabla f(x_k) +  J_F(x_k)^T \lambda_k = 0$, which implies that \eqref{kktt1}  holds. 

\medskip 

\noindent It remains to consider $\|s_k^*\| = \epsilon$. Then first-order conditions for $s_k^*$ imply that there exist $\lambda_k = \rho\sign(F(x_k)+J_F(x_k)s_k^*)\circ |F(x_k)+J_F(x_k)s_k^*|^{q-1}$, $u_k^* \geq 0$, and $z_k^* \in \partial (\|s_k^*\|)$ such that
\begin{equation} \label{eqref11}
    \nabla f(x_k) + J_F(x_k)^T \lambda_k + u_k^*z_k^* = 0.
\end{equation}
It follows from the definition of $\Psi_r(x_k)$ that
\[
\begin{aligned}
\Psi_{\epsilon}(x_k) &= \bar{\mathcal{P}}^q_\rho(x_k;x_k) - \bar{\mathcal{P}}^q_\rho(x_k+s_k^*;x_k) \\
&= -\langle\nabla f(x_k), s_k^*\rangle + \frac{\rho}{q} \left(\|F(x_k)\|_q^q - \|F(x_k)+J_F(x_k)s_k^*\|_q^q\right),
\end{aligned}
\]
and replacing $\nabla f(x_k)$ from \eqref{eqref11} into the above, we deduce:
\begin{align}\label{eq3.6}
&\Psi_{\epsilon}(x_k) = \frac{\rho}{q} \left(\|F(x_k)\|_q^q - \| F(x_k)+J_F(x_k)s_k^*\|_q^q\right) + \langle s_k^*, J_F(x_k)^T \lambda_k \rangle + u_k^*\langle s_k^*, z_k^*\rangle \nonumber  \\
&= \frac{\rho}{q}\left(\|F(x_k)\|_q^q - \| F(x_k)+J_F(x_k)s_k^*\|_q^q\right) + \langle s_k^*, J_F(x_k)^T \lambda_k \rangle + u_k^*\times \epsilon, 
\end{align}
where we also used that $\langle s_k^*, z_k^*\rangle = \|s_k^*\|=\epsilon$. Let $\varphi(s) = \frac{\rho}{q}\|F(x_k)+J_F(x_k)s \|_q^q$, which is convex; then $\varphi(0) - \varphi(s_k^*) \geq (-s_k^*)^T J_F(x_k)^T \lambda_k$, where $\lambda_k =\rho\sign(F(x_k)+J_F(x_k)s_k^*) \circ|F(x_k)+J_F(x_k)s_k^*|^{q-1}$. We then deduce that:
\[
\frac{\rho}{q} \left(\|F(x_k)\|_q^q - \| F(x_k)+J_F(x_k)s_k^*\|_q^q\right) + (s_k^*)^T J_F(x_k)^T \lambda_k \geq 0,
\]
and thus from \eqref{eq3.6} and the fact that $\Psi_{\epsilon}(x_k) \leq \epsilon^2$, we have:
\[
\epsilon^2 \geq \Psi_{\epsilon}(x_k) \geq u_k^*\times\epsilon.
\]
From \eqref{eqref11} and $\|z_k^*\| = 1$, we deduce
\[
u_k^* = u_k^*\|z_k^*\| = \|\nabla f(x_k) + J_F(x_k)^T \lambda_k\|\leq \epsilon.
\]
Hence, \eqref{kktt1} holds with $\lambda_k = \rho \sign(F(x_k)+J_F(x_k)s_k^*) |F(x_k)+J_F(x_k)s_k^*|^{q-1}$.
 Now, let us consider a KKT point $x^*\in\mathcal{S}$. LICQ ensures the existence of a corresponding $y^*$ such that:
\[
\nabla f(x^*)+{J_F(x^*)}^Ty^*=0  \hspace{0.3cm}\text{and}\hspace{0.3cm} F(x^*)=0.
\]
Let us analyze how much $\lambda_k$ deviates from a Lagrange multiplier $y^*$.  
We have:
\[
y^*=-{J_F(x^*)}^{+}\nabla f(x^*).
\]
Moreover, considering:
\[
 \|\nabla f(x_k) + J_F(x_k)^T \lambda_k\|\leq \epsilon,
\]
it then follows that there exists a vector $d\in \mathbb{R}^n$ with $\|d\|\leq 1$ such that:
\[
\nabla f(x_k) + J_F(x_k)^T \lambda_k=\epsilon d.
\]
This implies:
\[
\lambda_k=-J_F(x_k)^+\nabla f(x_k)+ \epsilon J_F(x_k)^+d.
\]
Hence:
\begin{align*}
    \|\lambda_k-y^*\|&=\|-J_F(x_k)^+\nabla f(x_k)+J_F(x^*)^+\nabla f(x^*)+\epsilon J_F(x_k)^+d\|\\
    &\leq \|J_F(x^*)^+\nabla f(x^*)-J_F(x_k)^+\nabla f(x_k)\|+\epsilon\| J_F(x_k)^+d\|.
\end{align*}
Given the continuity of $J_F^+$ and $\nabla f$, along with the fact that $x_k$ and $x^*$ belong to the compact set $\mathcal{S}$ and that $\|d\|\leq 1$, we conclude that there exists a constant $M_1\geq 0$ such that:
\begin{align*}
    \|\lambda_k-y^*\|&\leq M_1.
\end{align*}
Then:
\[
  \|\lambda_k\|\leq M_1+\|y^*\|.
\]
Consequently:
\begin{equation}\label{feasiblity1}
    \|F(x_k)+J_F(x_k)s_k^*\|=\frac{\||\lambda_k|^{\frac{1}{q-1}}\|}{\rho^{\frac{1}{q-1}}}\leq \left(\frac{\|\lambda_k\|_q}{\rho}\right)^{\frac{1}{q-1}} \leq \mathcal{O}(\epsilon).
\end{equation}
It then follows that:
\begin{align*}
    \|F(x_k)\|&\leq \|F(x_k)+J_F(x_k)s_k^*\|+\|J_F(x_k)s_k^*\|\\
    &\leq \|F(x_k)+J_F(x_k)s_k^*\|+M_F\|s_k^*\|\leq\mathcal{O}(\epsilon)+M_F\|s_k^*\|.
\end{align*}
Moreover, since $\|s_k^*\|\leq \epsilon$, we get:
\[
 \|\nabla f(x_k) + J_F(x_k)^T \lambda_k\|\leq \epsilon \text{ and } \|F(x_k)\|\leq\mathcal{O}(\epsilon),
\]
after $K=\mathcal{O} \left(\frac{\rho^{\frac{1}{q}}}{\epsilon^2} \right)=\mathcal{O} \left(\frac{1}{\epsilon^{2+\frac{q-1}{q}}} \right)$ iterations. 
\end{proof}
\begin{remark}\label{remark_impo1}
In addition to the discussion in Remark \ref{remark_impo}, it is worth noting that, from the convergence rate to an $\epsilon$-first-order solution of problem \eqref{eq1} which is of order $\mathcal{O}\left(\frac{1}{\epsilon^{2+\frac{q-1}{q}}}\right)$ outer iterations {(i.e., number of functions, gradients and Jacobians evaluations)},  as \( q \) approaches values close to 1, the \red{convergence rate} of Algorithm \ref{alg1} becomes \(\mathcal{O}\left(\frac{1}{\epsilon^2}\right)\), which coincides with the standard \red{convergence rate} for exact penalty methods \cite{CarGou:11, NabNec:25}. On the other hand, when \( q = 2 \), the \red{convergence rate} becomes \(\mathcal{O}\left(\frac{1}{\epsilon^{2.5}}\right)\), matching the rate \red{of} quadratic penalty methods established {e.g.,} in \cite{ElbNec:25}.
\end{remark}

\begin{remark}\label{remark_total}
For the total complexity (inner and outer iterations), note that the subproblem in Step 5 of our algorithm is unconstrained with a strongly convex objective function whose gradient is locally Hölder continuous with exponent \( q - 1 \). According to \cite{DevGli:13}, solving this subproblem to an accuracy \( \epsilon_{\text{sub}} = \epsilon \) using an accelerated gradient method requires:
\[
\mathcal{O}\left(\frac{\rho^{\frac{2}{3q-2}}}{\ubar{\beta}^{\frac{q}{q-1}} \epsilon^{\frac{2-q}{3q-2}}}\right)
\]  
gradient steps {(i.e., products of the form Jacobian at $x_k$ times vectors)}, up to a logarithmic factor. By selecting \( \rho = \mathcal{O}\left(\frac{1}{\epsilon^{q-1}}\right) \) and \( \ubar{\beta} = \mathcal{O}\left(\rho^{\frac{1}{q}}\right) = \mathcal{O}\left(\frac{1}{\epsilon^{\frac{q-1}{q}}}\right) \), {the worst-case} complexity of solving the subproblem simplifies to  
$\mathcal{O}\left(\frac{1}{\epsilon^{\frac{1}{3q-2}}}\right).$  
Thus, the total complexity for obtaining an \( \epsilon \)-first-order solution of problem \eqref{eq1} is  
$\mathcal{O}\left(\frac{1}{\epsilon^{2+\frac{q-1}{q}+\frac{1}{3q-2}}}\right)$.
In particular, as \( q \to 1 \), this complexity approaches \( \mathcal{O}\left(\frac{1}{\epsilon^3}\right) \), while for \( q = 2 \), it reduces to \( \mathcal{O}\left(\frac{1}{\epsilon^{2.75}}\right) \). The optimal total complexity is \( \mathcal{O}\left(\frac{1}{\epsilon^{2.74}}\right) \) and it is achieved for \( q = 1 + \frac{1}{\sqrt{3}}\approx 1.33 \).   Paper  \cite{CarGou:11}, which employs a Lipschitz penalty function approach,  has a total complexity of order \(\mathcal{O}\left(\frac{1}{\epsilon^{3}}\right)\),  when employing the scheme in \cite{DevGli:13}, which is usually higher  than our total complexity for any $q\in(1,2]$.
\end{remark}

 \noindent  Thus, apart from its straightforward implementation, our algorithm holds a compelling  theoretical convergence guarantee. One of its prominent strengths is the avoidance of complex subroutines. Unlike \cite{CarGou:11}, where the subproblem is non-differentiable and \cite{LinMa:22}, where the subproblem is nonconvex, our method simplifies  to solving an unconstrained  differentiable strongly convex subproblem. With a global theoretical convergence assurance to reach an $\epsilon$-first-order solution in at most $\mathcal{O}(\frac{1}{\epsilon^{2+\frac{q-1}{q}}})$ outer iterations, our approach ensures the dependable exploration of optimal solutions for a wide spectrum of nonconvex optimization problems, even those with nonlinear constraints, a capability absent in \cite{KonMel:18}, which only handles linear equality constraints. These  properties  position our algorithm as an efficient and versatile tool in the realm of nonlinear optimization, making it an appealing choice for various practical applications.


\subsection{Selection of the penalty parameter $\rho$}
\noindent The previous results, deriving the total number of iterations required to reach an $\epsilon-$ first order  solution to the problem, are based on the assumption that the penalty parameter $\rho$ exceeds a particular threshold. However, determining this threshold in advance presents difficulties, as it relies on unknown parameters associated with the problem functions and algorithm parameters.
To meet this challenge, we propose a methodology for determining a sufficiently large value of $\rho$ without requiring precise parameter information. Inspired by Algorithm 3 of \cite{XieWri:21}, our approach consists in iteratively calling Algorithm \ref{alg1} as an inner loop. If Algorithm \ref{alg1} fails to converge within a predefined number of iterations, we progressively increment the penalty parameter $\rho$ by a constant factor in the outer loop. A full description of this approach can be found in algorithm \ref{alg2}.
In some applications, such as model predictive control, it is necessary to handle a series of problems with a specified $\epsilon$ level of accuracy. In particular, the optimal penalty parameter for the initial optimization problem often remains effective for subsequent problems. This suggests that dedicating more computation resources to the initial optimization problem can result in substantial savings for the remaining optimization problems. This adaptive approach therefore guarantees the effectiveness and efficiency of our algorithm, even in scenarios where precise parameter information is not available. This adaptability makes our method a practical and robust solution for a wide range of optimization problems, potentially offering significant advantages in a variety of real world applications.

\begin{algorithm}
\caption{qLP method with trial value of $\rho$}\label{alg2}
\begin{algorithmic}[1]
\State  $\textbf{Initialization: } \tau>1, \epsilon>0 \text{ and } \rho_0>0;$
\State $k \gets 0$
\While{$\text{ Infeasible }$}
     \State $\rho_{k+1}\gets \tau \rho_k$
    \State Call Algorithm \ref{alg1} using warm start
    \State $k \gets k+1$
\EndWhile
\end{algorithmic}
\end{algorithm}


\section{Numerical results}\label{sec5}
\noindent In this section, we conduct a numerical comparison of Algorithm \ref{alg1} (qLP) for various values of $q$, specifically, for $q=1.001, q=1.5,$ and $q=2$, as well as with Algorithm 2.1 from \cite{CarGou:11}. These comparisons are carried out on nonconvex optimization problems with nonlinear equality constraints selected from the CUTEst collection. The simulations are implemented in Python and executed on a PC with a CPU running at 2.70GHz and equipped with 16GB of RAM. For the implementation of our method, we employ the following stopping criteria: we terminate the algorithms when the difference between two consecutive values of the objective function falls below a tolerance of $\epsilon_1=10^{-3}$, and the norm of the constraints is less than $\epsilon_2=10^{-5}$. To assess the sensitivity of our algorithm to parameters choice \( \rho \) and \( \beta \), we test several values for these parameters. We choose for $\rho$ three different values for each problem and then some corresponding $\beta$. If an algorithm does not solve a specific problem  in 30 min., we consider that problem unsolved by the algorithm and mark the corresponding entry in Table 1 as "-". The results are presented in Table \ref{tab1}.

\begin{table}
\small
{
    
      \centering
\begin{adjustbox}{width=\columnwidth,center}

    \begin{tabular}{|c|c|cc|cc|cc|cc|}
    \hline
   \multirow{3}{*}{(n,m)}& \multirow{3}{*}{\backslashbox{$(\rho,\beta)$}{Alg}}
     & \multicolumn{2}{c|}{$q=1.001$} &
     \multicolumn{2}{c|}{$q=1.5$} &
      \multicolumn{2}{c|}{$q=2$} &
      \multicolumn{2}{c|}{Alg. 2.1 \cite{CarGou:11}} \\ \cline{3-10}
              && \# iter     & cpu  
          & \# iter    & cpu &
           \# iter     & cpu &  
           \# iter     & cpu \\ 
        & & $f^*$    & $\|F\|$  
          & $f^*$    & $\|F\|$ &
           $f^*$    & $\|F\|$  &
           $f^*$    & $\|F\|$   \\
    \hline
  
 \multirow{3}{*}{DTOC4}
       & $\rho=10^2$ &  3 &  \textbf{2.84} 
    & - &  - &  
      - & - &
       3 & 4.86\\
    &$\beta=1$& 2.95 & 1.76e-06
      &- & - &  
      -& - &
      2.95& 3.65e-07\\ \cline{2-10}
         &$\rho=10^5$&  4 &  9.44 
    &4 &  11.56 &  
      - & - &
       3 & 12.32\\
       \multirow{3}{*}{(147,98)}
    &$\beta=5$& 2.95 & 2.3e-08
      &2.95 & 5.09e-07 &  
      -& - &
      2.95 & 1.29e-07\\ \cline{2-10}
        &$\rho=10^7$ &  3 &  9.21
    &5 &  32.22 &  
      13 & 3.15 &
      4 & 14.88\\
    &$\beta=10$& 2.95 & 1.45e-08
      &2.95 & 5.85e-07 &  
      2.95& 2.46e-06 &
      2.95 & 9.87e-08\\\hline
      
\multirow{3}{*}{DTOC4 }
 & $\rho=10^3$&  3 &  5.02 
    &- &  - &  
      - & - &
      6 & 5.89 \\
    &$\beta=1$& 2.88 & 3.7e-06
      &- & - &  
      -& - &
       2.88 & 1.38e-06\\ \cline{2-10}
       & $\rho=10^5$ &  3 &  38.63
    &3 &  65.21 &  
      - & - &
       8 & 35.07\\
      \multirow{3}{*}{(297,198)}
    &$\beta=5$& 2.88 & 3.69e-08
      &2.88 & 3.35e-08 &  
      -& - &
      2.88 & 3.64e-08 \\ \cline{2-10}
    & $\rho=10^7$ &  3 &  72.19
    &3 &  98.56 &  
      7 &  \textbf{3.47} &
      8 & 65.54\\
    &$\beta=10$& 2.88 & 9.43e-08
      &2.88 & 1.48e-07 &  
      2.88& 8.26e-06 &
      2.88 & 1.33e-08\\\hline

    \multirow{3}{*}{DTOC5}
       & $\rho=10^2$ &  9 &  \textbf{5.83}
    &- &  - &  
      - & - &
       29 & 20.06\\
       & $\beta=1$& 1.53 & 5.11e-06
      &- & - &  
      -& - &
      1.54 & 7.09e-07\\ \cline{2-10}
            & $\rho=10^5$&  9 &  34.52
    &7 &  23.88 &  
      - & - &
       25 & 44.81 \\
        \multirow{3}{*}{(98,49)}
       & $\beta=5$ & 1.53 & 3.85e-07
      &1.53 & 2.04e-06 &  
      -& - &
       1.53 & 2.22e-07 \\ \cline{2-10}
    & $\rho=10^7$ &  8 &  63.35
    &9 &  32.11 &  
      17 & 6.12 &
      26 & 65.80\\
       & $\beta=10$ & 1.53 & 3.42e-08
      &1.53 & 5.00e-07 &  
      1.53& 7.76e-06 &
      1.53 & 3.89e-08\\\hline
      
   \multirow{3}{*}{DTOC5 }
     & $\rho=10^2$&  6 & 12.59
        &- &  - &  
      - & - &
       10 & 11.65\\
     & $\beta=1$ & 1.53 & 7.05e-06
      &- & - &  
      - & - &
      1.56 & 1.03e-07\\ \cline{2-10}
           & $\rho=10^5$ &  6 &  51.53
        &5 &  41.62 &  
      - & - &
       10& 36.07\\
      \multirow{3}{*}{(198,99)}
     & $\beta=5$ & 1.53 & 7.73e-07
      &1.53 & 5.64e-06 &  
      -& - &
      1.53 & 7.51e-08\\ \cline{2-10}
   & $\rho=10^7$ &  4 &  103.27
        &23 &  143.79 &  
      21 &   \textbf{7.33} &
       11 & 87.34\\
     & $\beta=10$ & 1.53 & 3.18e-08 
      &1.53 & 6.04e-07 &  
      1.53& 5.88e-06 &
      1.53 &2.68e-08 \\\hline

   \multirow{3}{*}{DTOC6}
   &$\rho=10^3$ &  15 &  14.34
        &- &  - &  
      - & - &
      16 & 16.46\\
     & $\beta=4$& 727.98 & 2.90-07
      &- & - &  
      -& - & 
      727.98&4.22e-07 \\ \cline{2-10}
       &$\rho=10^6$  &  15 &  20.25
        &22 &  59.24 &  
      - & - &
      20 & 42.97\\
      \multirow{3}{*}{(140,60)}
     & $\beta=10$& 727.98 & 2.99-07
      &727.98 & 1.88e-06 &  
      -& - & 
      727.98& 8.73e-08\\ \cline{2-10}
   &$\rho=10^9$ &  15 &  79.40
        & 12 &  42.75 &  
      21 &  \textbf{9.26}&
      17 & 81.58 \\
     & $\beta=50$ & 727.98 & 3.16-08
      &727.98 & 7.06e-07 &  
      727.98& 1.23e-06 & 
      727.98 & 4.21e-08 \\\hline
      
    \multirow{3}{*}{DTOC6 }
      &$\rho=10^4$&  18 & 17.82 
    &- &  - &  
      - & - &
      34  & 23.67\\
    &$\beta=4$& 6846.61 & 5.23e-07
      &- & - &  
      - & - &
     6846.61 & 8.95e-06\\ \cline{2-10}
          &$\rho=10^7$&  18 &  61.77
    &23 &  73.53 &  
      - & - &
       33 & 57.12\\
       \multirow{3}{*}{(200,100)}
   & $\beta=20$& 6846.61 & 7.95e-08
      &6846.61 & 1.46e-06 &  
      - & - &
      6846.61 & 3.42e-06 \\ \cline{2-10}
    &$\rho=10^{10}$&  18 &  112.34
    & 22 &  123.15 &  
      36 &  \textbf{11.86 } &
       33 & 95.74\\
    &$\beta=50$& 6846.61 & 2.55e-08
      & 6846.61 & 4.66e-07 &  
      6846.61 & 5.24e-06 &
      6846.61 &7.27e-07 \\\hline

   \multirow{3}{*}{ORTHREGA }
    & $\rho=10^2$&  12 &  \textbf{7.03}
    & - &  - &  
      - & - &
       - & -\\
    &$\beta=1$& 414.53 & 1.93e-06
      &- & - &  
      -& - &
       -& -\\\cline{2-10}
         & $\rho=10^5$&  10 &  74.88
    & 46 &  51.63 &  
      - & - &
       17 & 27.11\\
       \multirow{3}{*}{(133,64)}
    &$\beta=5$& 414.53 & 3.08e-07
      &414.53 & 1.25e-06 &  
      -& - &
      414.53 & 8.48e-06 \\\cline{2-10}
   & $\rho=10^7$&  13 &  96.32
    & 46 &  129.09 &  
      47 & 7.40 &
       19 & 73.83\\
    &$\beta=10$& 414.53 & 7.24e-08
      &414.53 & 5.42e-07 &  
      414.53 & 6.19e-06 &
      414.53 & 6.96e-07\\\hline
      
     \multirow{3}{*}{ORTHREGA}
       & $\rho=10^2$ &  27 &  23.99
    &- &  - &  
      - & - &
       - & - \\
    &$\beta=1$ & 1664.80 & 1.91e-06
      &- & - &  
      -& - &
       - & - \\ \cline{2-10}
     & $\rho=10^5$&  25 &  83.31
    &49 &  112.57 &  
      - & - &
      42 & 78.39\\
       \multirow{3}{*}{(517,256)}
    &$\beta=5$ & 1664.80 & 3.19e-07
      &1664.80 & 2.07e-06 &  
      - & - &
      1664.80 & 3.67e-06\\ \cline{2-10}
   & $\rho=10^8$ &  27 &  107.78
    & 49 &  153.02&  
      103 & \textbf{14.24} &
      41 & 124.16\\
    &$\beta=10$ & 1664.80 & 4.87e-08
      &1664.80 & 7.63e-07 &  
      1664.80& 9.25e-06 &
        1664.80& 5.04e-07  \\
      

       \hline
      
    \end{tabular}%
    \end{adjustbox}
}
\caption{Comparing the performance of qLP with different values of $q$ (i.e., $q=1.001, q=1.5, q=2$) and Algorithm 2.1 from \cite{CarGou:11} in solving various problems from the CUTEst collection.}
\label{tab1}
\end{table}

\medskip 

\noindent Table 1 provides the number of iterations, CPU time (in seconds), objective values, and feasibility violations for several test cases.
It's worth noting that when testing our algorithm on problems from CUTEst, as $q$ approaches 1, the penalty parameter for qLP tends to converge to a finite value. This observation validates our theory. In fact, when $q$ is very close to 1, as in the case of $q=1.001$, we may discover a finite penalty parameter $\rho$ that can ensure feasibility. This parameter may be smaller than the penalty parameter $\rho$ required for feasibility in the case of the exact penalty method in \cite{CarGou:11}. An illustrative example of this behavior can be seen in the problem ORTHREGA, from Table \ref{tab1}, where for $\rho=10^2$, Algorithm 2.1 from \cite{CarGou:11} struggles to find a feasible solution. This is because, as $q$ approaches 1, the norm $\|\cdot\|_q^q$ (utilized in qLP) becomes significantly sharper than $\|\cdot\|$ (used in Algorithm 2.1 from \cite{CarGou:11}). Furthermore, qLP remains the most efficient in terms of number of iterations when $q$ is close to 1 (and sometimes also in terms of CPU time), provided that the optimal penalty parameter $\rho$ is selected. One can see that our method is robust with respect to $q$. Typically, the best $\rho$ is the smallest one that ensures feasibility.


\section{Conclusions}
\label{sec6}
\noindent In this paper, we have introduced a new approach called the $\ell_q$ linearized penalty method (qLP) to deal with optimization problems involving nonconvex objective and nonlinear equality constraints, in particular those exhibiting local smoothing. This method involves linearizing the cost function and functional constraints within the $\ell_q$ penalty function with a regularization term. By dynamically determining the (proximal) regularization parameter, we establish the convergence rate to reach an $\epsilon$-first-order  optimal solution.
In addition, we conduct numerical experiments to demonstrate the effectiveness of the proposed algorithm. 



\section*{Conflict of interest}
The authors declare that they have no conflict of interest.

\section*{Data availability}
It is not applicable.

\section*{Acknowledgments}
\noindent The research leading to these results has received funding from: the European Union's Horizon 2020 research and innovation programme under the Marie Sklodowska-Curie grant agreement no. 953348.



\section*{Appendix}
\noindent \textbf{Proof of Lemma \ref{holder}}
Using basic operations, we have:
\begin{align*}
   |\sign(x)|x|^{\nu}-\sign(y)|y|^{\nu}| &= |(\sign(x)-\sign(y))|x|^{\nu}+\sign(y)(|x|^{\nu}-|y|^{\nu})| \\
   &\leq |\sign(x)-\sign(y)||x|^{\nu}+|\sign(y)|||x|^{\nu}-|y|^{\nu}|
\end{align*}
Now, let us first prove that:
\[
|\sign(x)-\sign(y)||x|^{\nu}\leq 2|y-x|^{\nu}.
\]
To do that, we will distinguish the following cases:

\noindent \textbf{Case 1:} $x=0$ or $y=0$ or both $x$ and $y$ have the same sign; this is an obvious case.

\noindent \textbf{Case 2:} $x$ and $y$ have different signs:
\[
|\sign(x)-\sign(y)||x|^{\nu}= 2|x|^{\nu}.
\]
Since $|\cdot|^{\nu}$ is a nondecreasing function on $[0,\infty)$ for any $\nu>0$ and since $0\leq|x|\leq|y-x|$ (keeping in mind that $x$ and $y$ have different signs), then we have:
\begin{equation}\label{nice_inequality}
|x|^{\nu}\leq|y-x|^{\nu}    
\end{equation}
Hence, we get:
\[
|\sign(x)-\sign(y)||x|^{\nu}\leq 2|y-x|^{\nu}.
\]
Now we need to prove that:
\[
|\sign(y)| \cdot ||x|^{\nu}-|y|^{\nu}|=||x|^{\nu}-|y|^{\nu}|\leq  |y-x|^{\nu}.
\]
As before, we need to separate the following cases:\\
\noindent \textbf{Case 2.1:} $x=0$ or $y=0$ or both $x$ or $x=y$; this is an obvious case.

\noindent \textbf{Case 2.2:} $x$ and $y$ have different signs:
\[
|x-y|\geq\max\{|x|,|y|\}>0.
\]
Since $|\cdot|^{\nu}$ is symmetric with respect to the vertical axis for any $\nu>0$ and it is nondecreasing on $[0,\infty)$, hence:
\[
|x-y|^{\nu}\geq\max\{|x|^{\nu},|y|^{\nu}\}\geq\max\{|x|^{\nu},|y|^{\nu}\}-\min\{|x|^{\nu},|y|^{\nu}\}=||x|^{\nu}-|y|^{\nu}|.
\]

\noindent \textbf{Case 3:} $x$ and $y$ have the same sign:
Since $|\cdot|^{\nu}$ is symmetric for any $\nu>0$ and it is concave on the intervals $(-\infty,0]$ and $[0,+\infty)$, then with the loss of generality, we can assume that $x>y>0$. Now by applying the concavity of $|\cdot|^{\nu}$ on $[0,\infty)$, it follows that:
\begin{align}\label{x_eq}
    |x|^{\nu}&=|\frac{|y|}{|x|+|y|}\times0+\frac{|x|}{|x|+|y|}\times|x+y||^{\nu}\nonumber\\
    &\geq \frac{|y|}{|x|+|y|} \times |0|^{\nu}+\frac{|x|}{|x|+|y|}\times|x+y|^{\nu}\nonumber\\
    &=\frac{|x|}{|x|+|y|}\times|x+y|^{\nu}.
\end{align}
Similarly, we have:
\begin{align}\label{y_eq}
    |y|^{\nu}&=|\frac{|x|}{|x|+|y|}\times0+\frac{|y|}{|x|+|y|}\times|x+y||^{\nu}\nonumber\\
    &\geq \frac{|x|}{|x|+|y|} \times |0|^{\nu}+\frac{|y|}{|x|+|y|}\times|x+y|^{\nu}\nonumber\\
    &=\frac{|y|}{|x|+|y|}\times|x+y|^{\nu}.
\end{align}
Summing up \eqref{x_eq} and \eqref{y_eq} results in the following:
\begin{equation}\label{concave}
    |x|^{\nu}+|y|^{\nu}\geq \frac{|x|}{|x|+|y|}\times|x+y|^{\nu} +\frac{|y|}{|x|+|y|}\times|x+y|^{\nu}=|x+y|^{\nu}.
\end{equation}
Now, we have:
\[
|x|^{\nu}=|x-y+y|^{\nu}
\]
and since $x-y>0$ and $y>0$, it follows from \eqref{concave} that:
\[
|x|^{\nu}=|x-y+y|^{\nu}\leq |x-y|^{\nu}+|y|^{\nu}.
\]
Hence from all cases, we have for any $x,y\in\mathbb{R}$:
\[
||x|^{\nu}-|y|^{\nu}|\leq |x-y|^{\nu}.
\]
Finally, we can write:
\[
 |\sign(x)|x|^{\nu}-\sign(y)|y|^{\nu}|\leq 3|y-x|^{\nu}.
\]
Which concludes our proof.\qed

\medskip
\noindent \textbf{Proof of Lemma \ref{yuan85}}
Starting from the definition of $\Psi(\cdot)$ in \eqref{eq:Psi}, we can express it as follows:
\[
\Psi(x) =  \bar{\mathcal{P}}^q_{\rho}(x;x) - \min_{\|s\|\leq 1}  \bar{\mathcal{P}}^q_{\rho}(x+s;x)= \max_{\|s\|\leq 1}  \{ \bar{\mathcal{P}}^q_{\rho}(x;x) - \bar{\mathcal{P}}^q_{\rho}(x+s;x)\}
\]
Since for $s=0$, we have:
\[
\bar{\mathcal{P}}^q_{\rho}(x;x) - \bar{\mathcal{P}}^q_{\rho}(x+s;x)=0,
\]
and hence:
\[
\Psi(x) = \max_{\|s\|\leq 1}  \{ \bar{\mathcal{P}}^q_{\rho}(x;x) - \bar{\mathcal{P}}^q_{\rho}(x+s;x)\}\geq 0.
\]
Now, let's consider the case where $\Psi(x) = 0$. It then follows that:
\[
0\in\arg \max_{\|s\|\leq 1}\{ \bar{\mathcal{P}}^q_{\rho}(x;x) - \bar{\mathcal{P}}^q_{\rho}(x+s;x)\}.
\]
Since $s=0$ is in the interior of the ball defined by $\{s\in\mathbb{R}^n\;:\; \|s\|\leq1\}$ and because  $ \bar{\mathcal{P}}^q_{\rho}(x;x) - \bar{\mathcal{P}}^q_{\rho}(x+s;x)$ is differentiable with respect to $s$, then:
\[
\nabla_s\left(\bar{\mathcal{P}}^q_{\rho}(x;x) - \bar{\mathcal{P}}^q_{\rho}(x+s;x)\right)\Big|_{s=0}=0
\]
Or, equivalently
\[
\nabla_s\left(\bar{\mathcal{P}}^q_{\rho}(x;x) - \bar{\mathcal{P}}^q_{\rho}(x+s;x)\right)\Big|_{s=0}=\nabla f(x)+\rho J_F(x)^T \sign(F(x))\circ|F(x)|^{q-1}=0.
\]
Hence, it follows that $x$ is a critical  point of the penalty function $\mathcal{P}^q_{\rho}(\cdot)$.
Which concludes our proof.\qed

\medskip

\noindent {\textbf{Proof of Lemma \ref{lemma3}}.
 Note that the subproblem's objective function $x \mapsto \bar{\mathcal{P}}^{q}_{\rho}(\cdot;x_k)+\frac{\beta_{k+1}}{2}\|x-x_k\|^2$ is strongly convex with strong convexity constant $\beta_{k+1}$. Combining this with  the optimality of $x_{k+1}$  and the fact that $\bar{\mathcal{P}}^q_{\rho}(x_{k};x_k)=\mathcal{P}^q_{\rho}(x_{k})$, we get: 
\begin{equation} \label{initial_decrease}
 \bar{\mathcal{P}}^q_{\rho}(x_{k+1};x_k)\leq \mathcal{P}^q_{\rho}(x_{k}) -\beta_{k+1}\|x_{k+1}-x_k\|^2.
\end{equation}

\vspace{-0.1cm}

\noindent Further, since $f$ has a Lipschitz gradient, we have the following.

\vspace{-0.3cm}

\begin{equation}\label{smooth_object}
f(x_{k+1}) - l_f(x_{k+1};x_k) \leq \frac{L_f}{2}\|x_{k+1}-x_k\|^2. 
\end{equation}

\vspace{-0.1cm}

\noindent Moreover, using the fact that the gradient of $\frac{\rho}{q}\|\cdot\|^q_q$ is  Holder continuous, we obtain: 
\begin{align}
   &\frac{\rho}{q}\|F(x_{k+1})\|^q_q - \frac{\rho}{q}\|l_F(x_{k+1};x_k)\|^q_q \nonumber\\
   & \overset{\eqref{holder_inequality}}{\leq}   \langle \rho \sign{(l_F(x_{k+1};x_k))}\circ|l_F(x_{k+1};x_k)|^{q-1}, F(x_{k+1})-l_F(x_{k+1};x_k) \rangle \nonumber\\
   & \quad + \frac{3\times m^{\frac{2-q}{2}}\rho}{q}\|F(x_{k+1})-l_F(x_{k+1};x_k)\|^q\nonumber\\
    & \leq   \langle \rho |l_F(x_{k+1};x_k)|^{q-1}, |F(x_{k+1})-l_F(x_{k+1};x_k)| \rangle  \nonumber\\
    & \quad + \frac{3\times m^{\frac{2-q}{2}}\rho}{q}\|F(x_{k+1})-l_F(x_{k+1};x_k)\|^q \nonumber\\
    & = \rho \left(\sum_{i=1}^{m}|l_{f_i}(x_{k+1};x_k)|^{q-1}\times |f_i(x_{k+1})-l_{f_i}(x_{k+1};x_k)|\right)^{q\times\frac{1}{q}}  \nonumber\\
   & \quad + \frac{3\times m^{\frac{2-q}{2}}\rho}{q}\|F(x_{k+1})-l_F(x_{k+1};x_k)\|^q\nonumber\\
     & \leq \rho \left(\sum_{i=1}^{m}|l_{f_i}(x_{k+1};x_k)|^{q}\right)^{\frac{q-1}{q}}\times \left( \sum_{i=1}^{m}|f_i(x_{k+1})-l_{f_i}(x_{k+1};x_k)|^q\right)^{\frac{1}{q}}  \nonumber\\
   & \quad + \frac{3\times m^{\frac{2-q}{2}}\rho}{q}\|F(x_{k+1})-l_F(x_{k+1};x_k)\|^q\nonumber\\
   & =    \rho \|l_F(x_{k+1};x_k)\|^{q-1}_q  \|F(x_{k+1})-l_F(x_{k+1};x_k) \|_q  \nonumber\\
   & \quad + \frac{3\times m^{\frac{2-q}{2}}\rho}{q}\|F(x_{k+1})-l_F(x_{k+1};x_k)\|^q \nonumber\\
   & \leq  \rho \|l_F(x_{k+1};x_k)\|^{q-1}_q \times m^{\frac{2-q}{2q}} \times \|F(x_{k+1})-l_F(x_{k+1};x_k) \|  \nonumber\\
   & \quad + \frac{3\times m^{\frac{2-q}{2}}\rho}{q}\|F(x_{k+1})-l_F(x_{k+1};x_k)\|^q \nonumber \\
    &\overset{\text{Ass. } \ref{assump2}}{\leq} \textstyle 
 \rho \|l_F(x_{k+1};x_k)\|_q^{q-1} \frac{m^{\frac{2-q}{2q}}\times L_F }{2}\|\Delta x_{k+1}\|^2 + \frac{3\times m^{\frac{2-q}{2}}\rho}{q} \left(\frac{L_F}{2}\|\Delta x_{k+1}\|^2\right)^q,  \label{decrease_feasib}
\end{align}
where the third inequality follows using Holder inequality, i.e., for any $x,y\in\mathbb{R}^m$ and for any 
\( (r, s) \in \mathbb{R}_{+} \):
\[
\left( \sum_{i=1}^{m} |x_i|^r |y_i|^s \right)^{r+s} 
\leq 
\left( \sum_{i=1}^{m} |x_i|^{r+s} \right)^r 
\left( \sum_{i=1}^{m} |y_i|^{r+s} \right)^s,
\]
and the forth inequality follows from the fact that $\|v\|_2\leq \|v\|_q \leq m^{\frac{2-q}{2q}}\|v\|_2, \quad \forall q\in (1,2]$ and $v\in\mathbb{R}^m$.
\noindent  Furthermore, we now apply Young's inequality for products, which states that for any  
\( r, s \in (1, \infty) \) satisfying the conjugate relation  $\frac{1}{r} + \frac{1}{s} = 1,$
the following inequality holds for all nonnegative scalars \( a, b \):  
\[
ab \leq \frac{a^r}{r} + \frac{b^s}{s}.
\]  
Using this for $r=q$ and $s=\frac{q}{q-1}$, we bound \( \rho \|l_F(x_{k+1};x_k)\|_q^{q-1} \) as follows:


\begin{align}
&\rho \|l_F(x_{k+1};x_k)\|_q^{q-1} - \frac{q-1}{q}\frac{ \beta_{k+1} - L_f}{L_F} \nonumber\\
&\leq\frac{\rho}{q^{2-q}}  \left(\frac{L_F}{ \beta_{k+1} - L_f}\right)^{q-1}  \left( \frac{\rho}{q} \| l_F(x_{k+1};x_k)\|_q^q\right)^{q-1}  \nonumber\\
&= \frac{\rho}{q^{2-q}}  \left(\frac{L_F}{ \beta_{k+1} - L_f}\right)^{q-1} \left( \bar{\mathcal{P}}^q_{\rho} (x_{k+1};x_k) - f(x_{k+1}) + f(x_{k+1}) - l_f(x_{k+1};x_k)  \right)^{q-1}  \nonumber\\
& \overset{\eqref{initial_decrease}, \eqref{smooth_object}}{\leq }\frac{\rho}{q^{2-q}}  \left(\frac{L_F}{ \beta_{k+1} - L_f}\right)^{q-1} \left( \mathcal{P}^q_{\rho}(x_k) - f(x_{k+1}) - \frac{2\beta_{k+1} - L_f}{2}\|\Delta x_{k+1}\|^2 \right)^{q-1}  \nonumber\\
& \leq \frac{\rho}{q^{2-q}}  \left(\frac{L_F}{ \beta_{k+1} - L_f}\right)^{q-1} \left( \mathcal{P}^q_{\rho}(x_k) - f(x_{k+1}) - \left(\beta_{k+1} - L_f\right)\|\Delta x_{k+1}\|^2 \right)^{q-1} \nonumber\\
&\overset{(f(x_{k+1})\geq \bar{f})}{\leq} \frac{\rho}{q^{2-q}}  \left(\frac{L_F}{ \beta_{k+1} - L_f}\right)^{q-1} \left( \left(\mathcal{P}^q_{\rho}(x_k) - \bar{f}\right) - \left(\beta_{k+1} - L_f\right)\|\Delta x_{k+1}\|^2 \right)^{q-1}  \nonumber \\
&\resizebox{\textwidth}{!}{$\leq \frac{3\times m^{\frac{(2-q)(q-1)}{2q}}\rho}{q^{2-q}}  \left(\frac{L_F}{ \beta_{k+1} - L_f}\right)^{q-1} \left( \left(\mathcal{P}^q_{\rho}(x_k) - \bar{f}\right) - \left(\beta_{k+1} - L_f\right)\|\Delta x_{k+1}\|^2 \right)^{q-1}$}  \nonumber \\
&\resizebox{\textwidth}{!}{$\leq \frac{3\times m^{\frac{(2-q)(q-1)}{2q}}\rho}{q^{2-q}}  \left(\frac{L_F}{ \beta_{k+1} - L_f}\right)^{q-1} \left(2^{2-q}\left( \mathcal{P}^q_{\rho}(x_k) - \bar{f} \right)^{q-1} - \left(\beta_{k+1} - L_f\right)^{q-1}\|\Delta x_{k+1}\|^{2(q-1)}\right)$}  \nonumber \\
&\leq \frac{3\times m^{\frac{(2-q)(q-1)}{2q}}\times2^{2-q}\rho}{q^{2-q}}  \left(\frac{L_F}{ \beta_{k+1} - L_f}\right)^{q-1} \left( \mathcal{P}^q_{\rho}(x_k) - \bar{f} \right)^{q-1} \nonumber\\
& \quad- \frac{3\times m^{\frac{(2-q)(q-1)}{2q}}\rho L_F^{q-1}}{q^{2-q}}\|\Delta x_{k+1}\|^{2(q-1)} \nonumber \\
& \overset{\eqref{eq_assu}}{\leq} \frac{\beta_{k+1}-L_f}{q L_F} - \frac{3\times m^{\frac{(2-q)(q-1)}{2q}}\rho L_F^{q-1}}{q^{2-q}}\|\Delta x_{k+1}\|^{2(q-1)}, \label{bound_grad_norm_squared}
\end{align}
where, the sixth inequality follows from the fact that for any \( a \geq b \geq 0 \) and any \( \nu \in (0,1] \), the following holds:  $(a-b)^{\nu} \leq 2^{1-\nu} a^{\nu} - b^{\nu}$,  where we set \( \nu = q-1 \in (0,1] \).
\noindent Using \eqref{bound_grad_norm_squared} in \eqref{decrease_feasib}, we get:
\begin{align}
  &\frac{\rho}{q}\|F(x_{k+1})\|^q_q - \frac{\rho}{q}\|l_F(x_{k+1};x_k)\|^q_q  \nonumber\\
  & \leq   \frac{\beta_{k+1}-L_f}{2} \|\Delta x_{k+1}\|^2 - 3\times m^{\frac{2-q}{2}}\rho L_F^q\left(\frac{1}{2 q^{2-q}} - \frac{1}{q 2^q} \right) \|\Delta x_{k+1}\|^4 \nonumber\\
  & \leq  \frac{\beta_{k+1}-L_f}{2} \|\Delta x_{k+1}\|^2. \label{to_use_next}
\end{align}
Moreover, we have:
\begin{align*} 
   & \mathcal{P}^q_{\rho}(x_{k+1}) - \bar{\mathcal{P}}^q_{\rho}(x_{k+1};x_k)\\
   &=f(x_{k+1}) - l_f(x_{k+1};x_k) + \frac{\rho}{q}\|F(x_{k+1})\|^q_q - \frac{\rho}{q}\|l_F(x_{k+1};x_k)\|^q_q.
\end{align*}
Using  \eqref{smooth_object} and \eqref{to_use_next} in the previous relation, it follows  (inequality \eqref{smoothness}):
\begin{equation*} 
     \mathcal{P}^q_{\rho}(x_{k+1}) \leq \bar{\mathcal{P}}^q_{\rho}(x_{k+1};x_k) + \frac{\beta_{k+1}}{2} \|\Delta x_{k+1}\|^2.
\end{equation*}
Finally, using \eqref{initial_decrease}, we get the decrease in \eqref{eq_descent}.
This proves our statement.\qed
}

\medskip

\noindent \textbf{Proof of Lemma \ref{bounded_regularization}}
Using the definition of $r_k$, $\bar{\Psi}(x_k, \beta_{k+1}) > 0$, and $\mathcal{P}^q_{\rho}(x_k) = \bar{\mathcal{P}}^q_{\rho}\left(x_k;x_{k}\right)$, it follows that $r_k \geq 1$, provided
\begin{equation*}
\mathcal{P}^q_{\rho}(x_{k+1}) - \bar{\mathcal{P}}^q_{\rho}\left(x_{k+1};x_{k}\right)-\frac{\beta_{k+1}}{2} \|x_{k+1}-x_k\|^2 \leq 0.
\end{equation*}
This inequality is guaranteed to hold by Algorithm \ref{alg1}, thereby completing the proof. \qed


\medskip

\noindent \textbf{Proof of Lemma \ref{bbound}}
We prove this result using induction arguments. Since $x_0,x_1\in\mathcal{S}$, then from Lemma \ref{lemma3} for $k=0$, it follows that: 
\begin{align}\label{ine111}
    f(x_1)+\frac{\rho}{q}\|F(x_1)\|_q^q+\frac{\beta_1}{2}\|x_1-x_0\|^2{\overset{\eqref{eq_descent}}{\leq}}\hspace{0cm} f(x_0)+\frac{\rho}{q}\|F(x_0)\|_q^q{\overset{{\eqref{ine10}}}{\leq}}\bar{\alpha}+c_0.
\end{align}
 Furthermore, exploiting the definition of $P_{k}$ for $k=1$, we have: 
\begin{align}\label{ine110}
\ubar{P}&\leq f(x_1)+\frac{\rho_0}{q}\|F(x_1)\|_q^q \nonumber\\
& {\overset{{(\rho\geq(q+1)\rho_0)}}{\leq}} P_1=\mathcal{P}^q_{\rho}(x_1)=f(x_1)+\frac{\rho}{q}\|F(x_1)\|_q^q{\overset{{\eqref{ine111}}}{\leq}}\bar{\alpha}+c_0=\bar{P}.
\end{align}
It then follows that for $k=1$, \eqref{important} is  verified. Now, assume that \eqref{important} holds for some $k\geq1$ (induction hypothesis) and we will prove that it continues to hold for $k+1$.
Since $x_{k}, x_{k+1}\in\mathcal{S}$ and from Lemma \ref{lemma3}, we have:
\[
    P_{k+1}-P_{k}\leq-\frac{\beta_{k+1}}{2}\|\Delta x_{k+1}\|^q\leq0.
\]
Together with the induction hypothesis, we obtain:
\[
    P_{k+1}\leq P_{k}{\overset{{}}{\leq}}\bar{P}.
\]
It remains to prove that $P_{k+1}\geq \ubar{P}$.  Using \eqref{lyapunov_function},  we have:
\begin{align*}
  P_{k}= f(x_{k})+\frac{\rho}{q}\|F(x_{k})\|_q^q\geq  f(x_{k})+\frac{\rho_0}{q}\|F(x_{k})\|_q^q{\overset{{(\text{Lemma } \ref{lem1})}}{\geq}}\ubar{P}. \label{bound}
\end{align*}
It follows that the sequence $\{P_{k}\}_{k\geq0}$ is bounded from below.
Finally, \eqref{important} is  proved, which completes our proof.\qed

\medskip
\medskip

\medskip

\medskip
\noindent \textbf{Proof of Lemma \ref{bounded_gradient}}
Using the optimality condition, we have:
\[
    \nabla f(x_{k})+{J_F(x_{k})}^T(\rho \sign(l_F(x_{k+1};x_k))\circ|l_F(x_{k+1};x_k)|^{q-1})+\beta_{k+1}\Delta x_{k+1}=0.
\]
Exploiting  definition of $\mathcal{P}^q_{\rho}$ and  properties of the derivative, it follows that:
\begin{align*}
    & \nabla\mathcal{P}^q_{\rho}(x_{k+1})\\
    =&\nabla f(x_{k+1})+{J_F(x_{k+1})}^T\big(\rho \sign(F(x_{k+1}))\circ|F(x_{k+1})|^{q-1}\big)\\
     = & \nabla f(x_{k+1})-\nabla f(x_k)-\beta_{k+1}\Delta x_{k+1}\\
     &+\rho\big({J_F(x_{k+1})}-{J_F(x_{k})}\big)^T\sign(F(x_{k+1}))\circ |F(x_{k+1})|^{q-1}\\
     & \resizebox{\textwidth}{!}{$+\rho{J_F(x_{k})}^T\big(\sign(F(x_{k+1})) \circ|F(x_{k+1})|^{q-1}-\sign(l_F(x_{k+1};x_k))\circ|l_F(x_{k+1};x_k)|^{q-1}\big).$}
\end{align*}
It then follows by applying the norm:
\begin{align*}
     &\|\nabla\mathcal{P}^q_{\rho}(x_{k+1})\|\\
     \leq & \|\nabla f(x_{k+1})-\nabla f(x_k)\|+\beta_{k+1}\|\Delta x_{k+1}\|\\
     &+\rho\|\sign(F(x_{k+1}))\circ |F(x_{k+1})|^{q-1}\|\|{J_F(x_{k+1})}-{J_F(x_{k})}\|\\
   &  \resizebox{\textwidth}{!}{$+  \rho  \|J_F(x_{k})\|\|\sign(F(x_{k+1}))\circ |F(x_{k+1})|^{q-1}-\sign(l_F(x_{k+1};x_k))\circ|l_F(x_{k+1};x_k)|^{q-1}\|$}\\
     \overset{\eqref{holder_inequality0}}{\leq}&\|\nabla f(x_{k+1})-\nabla f(x_k)\|+\beta_{k+1}\|\Delta x_{k+1}\|\\
     &+\rho\|\sign(F(x_{k+1})) \circ|F(x_{k+1})|^{q-1}\|\|{J_F(x_{k+1})}-{J_F(x_{k})}\|\\
     & +\rho\|{J_F(x_{k})}\| \left(3\times m^{\frac{2-q}{2}}\left\|F(x_{k+1}) - l_F(x_{k+1};x_k)\right\|^{q-1}\right)\\
     & = \|\nabla f(x_{k+1})-\nabla f(x_k)\|+\beta_{k+1}\|\Delta x_{k+1}\|\\
      &+\rho \left(\sum_{i=1}^{m}\left(|f_i(x_{k+1})|^2\right)^{q-1}\right)^{\frac{1}{2}}\|{J_F(x_{k+1})}-{J_F(x_{k})}\|\\
     & +\rho\|{J_F(x_{k})}\| \left(3\times m^{\frac{2-q}{2}}\left\|F(x_{k+1}) - l_F(x_{k+1};x_k)\right\|^{q-1}\right)\\
     \leq& \|\nabla f(x_{k+1})-\nabla f(x_k)\|+\beta_{k+1}\|\Delta x_{k+1}\|\\
      &+\rho \left(\sum_{i=1}^{m}\left(|f_i(x_{k+1})|^2\right)\right)^{\frac{q-1}{2}}\times \left(\sum_{i=1}^{m} 1 \right)^{\frac{2-q}{2}}\|{J_F(x_{k+1})}-{J_F(x_{k})}\| \\
      & +\rho\|{J_F(x_{k})}\| \left(3\times m^{\frac{2-q}{2}}\left\|F(x_{k+1}) - l_F(x_{k+1};x_k)\right\|^{q-1}\right)\\
       =& \|\nabla f(x_{k+1})-\nabla f(x_k)\|+\beta_{k+1}\|\Delta x_{k+1}\|\\
      &+m^{\frac{2-q}{2}}\rho\|F(x_{k+1})\|^{q-1} \|{J_F(x_{k+1})}-{J_F(x_{k})}\|\\
     & +\rho\|{J_F(x_{k})}\| \left(3\times m^{\frac{2-q}{2}}\left\|F(x_{k+1}) - l_F(x_{k+1};x_k)\right\|^{q-1}\right),
     \end{align*}
     where the second inequality follows from Hölder's inequality. Moreover, using the fact that for any vector \( v \in \mathbb{R}^m \) and \( q \in (1, 2] \), we have:
\[
\|v\|_2 \leq \|v\|_q,
\]
      we further get:
     \begin{align*}
     &\|\nabla\mathcal{P}^q_{\rho}(x_{k+1})\|\\
           \leq& \|\nabla f(x_{k+1})-\nabla f(x_k)\|+\beta_{k+1}\|\Delta x_{k+1}\|\\
      &+m^{\frac{2-q}{2}}\rho\|F(x_{k+1})\|_q^{q-1} \|{J_F(x_{k+1})}-{J_F(x_{k})}\|\\
     & +\rho\|{J_F(x_{k})}\| \left(3\times m^{\frac{2-q}{2}}\left\|F(x_{k+1}) - l_F(x_{k+1};x_k)\right\|^{q-1}\right)\\
       {\overset{\text{Ass. } \ref{assump2}}{\leq}}& \resizebox{\textwidth}{!}{$\Big(L_f+m^{\frac{2-q}{2}} L_F \rho\|F(x_{k+1})\|_q^{q-1}+\beta_{k+1}\Big)\|\Delta x_{k+1}\|+3m^{\frac{2-q}{2}}\rho M_F \left(\frac{L_F}{2}\|\Delta x_{k+1}\|^2\right)^{q-1}$}\\
     = & \resizebox{\textwidth}{!}{$\left(L_f+ m^{\frac{2-q}{2}} L_F \rho\|F(x_{k+1})\|_q^{q-1}+\beta_{k+1}\right)\|\Delta x_{k+1}\|+ \frac{3m^{\frac{2-q}{2}}\rho M_F L_F^{q-1}}{2^{q-1}}\|\Delta x_{k+1}\|^{2(q-1)},$}
    \end{align*}
 Furthermore, we have:
\begin{align*}
    \frac{\rho}{q}\|F(x_{k+1})\|^q_q &= \mathcal{P}^q_{\rho}(x_{k+1}) - f(x_{k+1})\\
    & \leq \mathcal{P}^q_{\rho}(x_{k+1}) - \bar{f}.
\end{align*}
It then follows that:
\[
\rho\|F(x_{k+1})\|_q^{q-1} \leq q^{\frac{q-1}{q}} \rho^{\frac{1}{q}} \left(\mathcal{P}^q_{\rho}(x_{k+1}) - \bar{f}\right)^{\frac{q-1}{q}} \overset{\eqref{eq_descent}}{\leq} q^{\frac{q-1}{q}} \rho^{\frac{1}{q}} \left(\mathcal{P}^q_{\rho}(x_{k}) - \bar{f}\right)^{\frac{q-1}{q}}.
\]
Therefore, we get:
\begin{align*}
     \|\nabla\mathcal{P}_{\rho}(x_{k+1})\| & \leq \left(L_f+ q^{\frac{q-1}{q}}m^{\frac{2-q}{2}} L_F \rho^{\frac{1}{q}}\left(\mathcal{P}^q_{\rho}(x_{k}) - \bar{f}\right)^{\frac{q-1}{q}} +\beta_{k+1}\right)\|\Delta x_{k+1}\| \\
     & \quad+ \frac{3m^{\frac{2-q}{2}}\rho M_F L_F^{q-1}}{2^{q-1}}\|\Delta x_{k+1}\|^{2(q-1)}
\end{align*}
This completes the proof.\qed
\medskip

\medskip
\noindent \textbf{Proof of Lemma \ref{eq:lemma_2.5}}
Let us first assume  that $\beta_{k+1}r^2 \leq \Psi_r(x_k)$. Then, we have:
\begin{align*}
\min_{s\in\mathbb{R}^n} \left\{\bar{\mathcal{P}}^q_{\rho}(x_k+s;x_{k}) + \frac{\beta_{k+1}}{2} \|s\|^2 \right\} &\leq \min_{\|s\|\leq r} \left\{ \bar{\mathcal{P}}^q_{\rho}(x_k+s;x_{k}) + \frac{\beta_{k+1}}{2} \|s\|^2 \right\} \\
&\leq \min_{\|s\|\leq r} \left\{ \bar{\mathcal{P}}^q_{\rho}(x_k+s;x_{k})\right\} + \frac{\beta_{k+1}r^2}{2}  \\
&\leq \min_{\|s\|\leq r} \left\{ \bar{\mathcal{P}}^q_{\rho}(x_k+s;x_{k})\right\} + \frac{\Psi_r(x_k)}{2} ,
\end{align*}
and so, from \eqref{eq:Psi_r} and \eqref{eq:Psi}, it follows that:
\[
\bar{\Psi}(x_k, \beta_{k+1}) \geq \bar{\mathcal{P}}^q_{\rho}(x_k;x_{k}) - \min_{\|s\|\leq r} \bar{\mathcal{P}}^q_{\rho}(x_k+s;x_{k}) - \frac{\Psi_r(x_k)}{2} = \Psi_r(x_k) - \frac{\Psi_r(x_k)}{2} = \frac{\Psi_r(x_k)}{2},
\]
which proves \eqref{coraliaa} in the case when $\beta_{k+1}r^2 \leq \Psi_r(x_k)$.\\
Now let $\beta_{k+1}r^2 > \Psi_r(x_k)$ and $s_k^* \triangleq \arg\min_{\|s\|\leq r}  \bar{\mathcal{P}}^q_{\rho}(x_k+s;x_{k})$. Then, by defining $s_k\triangleq x_{k+1}-x_k$, we get:
\begin{align*}
\bar{\mathcal{P}}^q_{\rho}(x_k+s_k;x_{k}) + \frac{\beta_{k+1}}{2} \|s_k\|^2 
&\leq \bar{\mathcal{P}}^q_{\rho}\left(x_k+\frac{\Psi_r(x_k)}{\beta_{k+1}r^2}s_k^*;x_{k}\right) +\frac{\beta_{k+1}}{2}\left\|\frac{\Psi_r(x_k)}{\beta_{k+1}r^2}s_k^*\right\|^2 \\
&\leq \bar{\mathcal{P}}^q_{\rho}\left(x_k+\frac{\Psi_r(x_k)}{\beta_{k+1}r^2}s_k^*;x_{k}\right) + \frac{\Psi_r(x_k)^2}{2\beta_{k+1}r^2},
\end{align*}
where, to obtain the second inequality, we used $\|s_k^*\| \leq r$. This and \eqref{eq:Psi_r} give
\begin{equation}
\label{eq:lemma_proof_step}
\bar{\Psi}(x_k, \beta_{k+1}) \geq  \bar{\mathcal{P}}^q_{\rho}\left(x_k;x_{k}\right) -  \bar{\mathcal{P}}^q_{\rho}\left(x_k+\frac{\Psi_r(x_k)}{\beta_{k+1}r^2}s_k^*;x_{k}\right)-  \frac{\Psi_r(x_k)^2}{2\beta_{k+1}r^2}.
\end{equation}
Using $0 < \frac{\Psi_r(x_k)}{\beta_{k+1}r^2} < 1$ and the fact that $\bar{\mathcal{P}}^q_{\rho}$  is convex, we obtain:
\[
\bar{\mathcal{P}}^q_{\rho}\left(x_k+\frac{\Psi_r(x_k)}{\beta_{k+1}r^2}s_k^*;x_{k}\right) \leq \left(1 - \frac{\Psi_r(x_k)}{\beta_{k+1}r^2}\right) \bar{\mathcal{P}}^q_{\rho}\left(x_k;x_{k}\right)  + \frac{\Psi_r(x_k)}{\beta_{k+1}r^2} \bar{\mathcal{P}}^q_{\rho}\left(x_k+s_k^*;x_{k}\right) ,
\]
which substituted into \eqref{eq:lemma_proof_step} gives
\begin{align*}
    \bar{\Psi}(x_k, \beta_{k+1}) &\geq \frac{\Psi_r(x_k)}{\beta_{k+1}r^2} \left(\bar{\mathcal{P}}^q_{\rho}\left(x_k;x_{k}\right) - \bar{\mathcal{P}}^q_{\rho}\left(x_k+s_k^*;x_{k}\right)\right) - \frac{\Psi_r(x_k)^2}{2\beta_{k+1}r^2} \\
    &= \frac{\Psi_r(x_k)^2}{\beta_{k+1}r^2} - \frac{\Psi_r(x_k)^2}{2\beta_{k+1}r^2} = \frac{\Psi_r(x_k)^2}{2\beta_{k+1}r^2},
\end{align*}
where we also used \eqref{eq:Psi} and the choice of $s_k^*$. This concludes our proof.\qed

\medskip

\noindent \textbf{Proof of Lemma \ref{lemma_criticality}}
It suffices to prove that given any $\epsilon \in (0, 1]$, the total number of iterations of Algorithm \ref{alg1} with $\Psi_{\epsilon}(x_k) > \epsilon^2$ is at most
\[
K \leq \left\lceil2\bar{\beta}\left(\bar{P}-\ubar{P}\right) \epsilon^{-2} \right\rceil \leq \mathcal{O}\left(\frac{\rho^{\frac{1}{q}} }{\epsilon^2}\right).
\]
Using Lemma \ref{eq:lemma_2.5} with the fact that  $\beta_k\leq \bar{\beta}$ for any $k\geq 1$, it follows that:
\[
\bar{\Psi}(x_k, \beta_{k+1}) \geq \frac{1}{2} \min\left(1, \frac{\Psi_{\epsilon}(x_k)}{\bar{\beta}\epsilon^2}\right)\Psi_{\epsilon}(x_k), \quad \text{for } k \geq 0.
          \]
Thus, while Algorithm \ref{alg1} does not terminate, $\Psi_{\epsilon}(x_k) > \epsilon^2$ and $\epsilon \leq 1$  provide
\[
\bar{\Psi}(x_k, \beta_{k+1}) \geq \frac{1}{2} \min\left(1, \frac{\epsilon^2}{\bar{\beta}\epsilon^2}\right) \epsilon^2 = \frac{\epsilon^2}{2\bar{\beta}},
\] 
where the equality follows from the fact that $\bar{\beta} \geq 1$. Combining the above inequality with \eqref{succesfull}, we get
\[
\mathcal{P}^q_{\rho}(x_k) - \mathcal{P}^q_{\rho}(x_{k+1}) \geq  \bar{\Psi}(x_k, \beta_{k+1}) \geq \frac{ \epsilon^2}{2\bar{\beta}}.
\]
Let $K>0$. Summing up the above inequality over $k$, we get
\[
\bar{P}-\ubar{P}\geq \sum_{k=0}^{K} [\mathcal{P}^q_{\rho}(x_k) - \mathcal{P}^q_{\rho}(x_{k+1})] \geq K \frac{\epsilon^2}{2\bar{\beta}},
\]
and so $K \leq 2\frac{\bar{\beta} (\bar{P}-\ubar{P})}{\epsilon^2}\leq \mathcal{O}\left(\frac{\rho^{\frac{1}{q}} }{\epsilon^2}\right)$, which proves our claim. \qed

\bigskip


\end{document}